\documentclass[10pt]{article}

\usepackage{bbm}
\usepackage{bbold}
\usepackage{amsmath}
\usepackage{amsthm}
\usepackage{amssymb} 
\usepackage{mathrsfs}
\usepackage{graphicx}
\usepackage{authblk}
\usepackage{framed}
\usepackage[backref=none,hypertexnames=false,colorlinks=true,urlcolor=blue,linkcolor=blue,citecolor=blue]{hyperref}
\usepackage[letterpaper,text={7in,9in},centering]{geometry}
\usepackage{comment}
\usepackage[table,xcdraw]{xcolor}
\usepackage[utf8]{inputenc}
\usepackage{graphicx}
\usepackage{float}
\usepackage{caption}
\usepackage{subcaption}
\usepackage{endnotes}

\usepackage[english]{babel}

\makeatletter\@addtoreset{equation}{section}\makeatother
\renewcommand{\theequation}{\arabic{section}.\arabic{equation}}

\newtheorem{thm}{Theorem}[section] 
\newtheorem{lem}[thm]{Lemma}  
 
\newtheorem{prop}[thm]{Proposition}

\newtheorem{rmk}[thm]{Remark}

   \def\ri{\texttt{i}}
\def\rd{{\rm d}} \def\re{{\rm e}} \def\fr{\mbox{$\frac{1}{2}$}}
\def\frr{\mbox{$\frac{1}{4}$}} \def\frrr{\mbox{$\frac{1}{6}$}} \def\R{\mathbb R}
  
 \def\qand{\quad\mbox{and}\quad}
 \def\uh{\widehat {u}} 
\def\hh{\widehat {h}}
\def\lth{{\langle\!\langle}} \def\lthb{{\Big\langle\!\!\Big\langle}}
\def\rth{{\rangle\!\rangle}} \def\rthb{{\Big\rangle\!\!\Big\rangle}}  \def\buhat{{\widehat{\bf u}}}

\def\iA{\mathscr{A}}
\def\iB{\mathscr{B}}

\def\eps{{\varepsilon}} 
   \def\xd{\mathrm{d}}
\def\langang{\langle\!\langle} \def\rangang{\rangle\!\rangle}

\renewcommand{\theequation}{\arabic{section}.\arabic{equation}}

\begin{document}

\title{Heteroclinic connections between finite-amplitude periodic orbits emerging from a codimension two singularity}
\author[1]{T.J. Bridges}
\author[1]{D.J.B. Lloyd}
\author[2]{D.J. Ratliff}
\author[3]{P. Sprenger}

\affil[1]{\small School of Mathematics and Physics, University of Surrey, Guildford, GU2 7XH, UK}
\affil[2]{\small Department of Maths, Phys and Elec Eng, Northumbria University, Newcastle NE1 8ST, UK}
\affil[3]{\small Department of Applied Mathematics, University of California at Merced, Merced, CA 95343, USA}

\date{}
\maketitle

\begin{abstract}
\noindent Heteroclinic connections between two distinct hyperbolic periodic orbits in
conservative systems are important in a wide range of applications. On the other hand, it is theoretically challenging to find large amplitude connections from scratch and compute them numerically. In this paper,
we use a codimension two singularity, in a family of periodic orbits, as an
organizing center for the emergence of heteroclinic connections.  A normal form
is derived whose unfolding produces two distinct finite amplitude periodic
orbits with an explicit heteroclinic connection. We also construct heteroclinic connections far from the singularity by numerical
continuation, using two numerical strategies: shooting and the core-farfield
decomposition.  A key geometric tool in the numerics is cylindrical foliations for the stable and unstable manifolds and their intersection. We introduce a new property of heteroclinic connections -- the
action -- and show it is an invariant along foliations, it has a jump at a surface of section, and it appears in a central way in the normal form theory.  We find
that the difference in asymptotic phase between minus and plus infinity is also
a key property. The theory is applied to the Swift-Hohenberg equation, the nonlinear
Schr\"odinger with fourth order dispersion, and coupled Boussinesq equations
from water waves, all of which have an energy and action
conservation law.  

\end{abstract}

\section{Introduction}

Heteroclinic connections in dynamical systems are solutions that connect two
distinct states and have applications in astrodynamics~\cite{hs22}, water waves~\cite{sbs23,rtb25}, pattern formation~\cite{kuw19}, and nonlinear optics~\cite{bgbk23}, for example. While heteroclinic
orbits involving equilibria (known as EtoE fronts) or an equilibrium and a
periodic orbit (known as EtoP fronts) in Hamiltonian systems are
well-studied~\cite{Beck2009,abcdssw19}, much less is known about periodic to
periodic connections (known as PtoP fronts) especially in conservative systems.
In this paper, we consider the theory and construction of heteroclinic
connections between two distinct finite-amplitude periodic states in
conservative systems; see Figure~\ref{fig-connection-schematic}  for an example
of the types of heteroclinic connections we are interested in. We present new
results on the role of action, the multiplicity of connections, the computation
of asymptotic phases, and introduce a normal form for the merger of two distinct
{\it finite-amplitude} hyperbolic periodic states (referred to as a codimension
two singularity herein). Unfolding the latter singularity gives explicit
connections, and a strategy for finding starting points of branches of
heteroclinic connections that are otherwise highly challenging to find. We develop a
numerical continuation method to find a singularity, and then use a second continuation strategy to compute large-amplitude connections far from the singularity.  The role of action contains several surprises
including its role in the codimension two singularity, its invariance on the
stable and unstable foliations, its appearance in Floquet theory, and a jump in
action that appears to be intrinsic to a connection. The main application of the theory is to nonlinear
waves and patterns, where these connections are fronts between two stationary or
periodic traveling waves, and so $x$ will feature as the time-like direction.

\begin{figure}[ht]
\begin{center}
\includegraphics[scale=1]{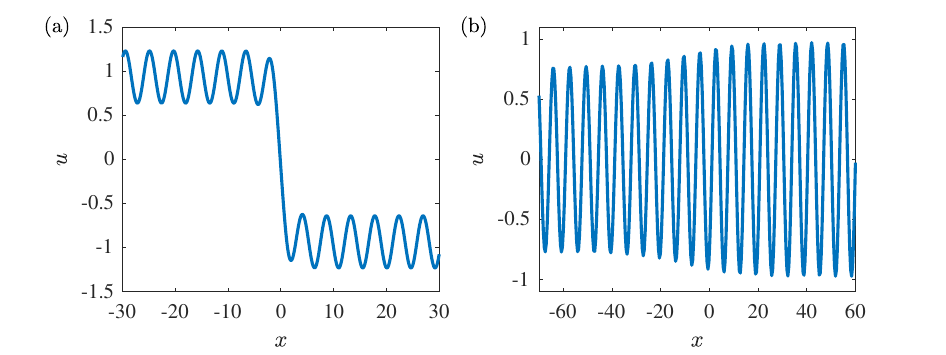}
\end{center}
\caption{Examples of a heteroclinic connection between two distinct periodic
  states (details of this example are in \S\ref{sec-SH357-review}). Panel (a) shows a heteroclinic
  connection between two periodic orbits in the cubic SH equation with
  $\sigma=1$ and panel (b) shows a heteroclinic connection in the SH357 equation with $\mu = 0.305$, $a = 1.438$, and $b = 2.117$. }
\label{fig-connection-schematic}
\end{figure}

There are a range of important developments in the theory and numerics of these
states in the literature that we take as starting points and build on. The
geometry of these heteroclinic connections along with novel numerics has been
reported in \textsc{Bandara et al.}~\cite{bgbk21,bgbk23,bgbk25} as well as
\textsc{Zhang et al.}~\cite{zkk12}, using Lin's method and a boundary-value
approach.  We will reproduce some of their results and extend them. The
multiplicity question (how many heteroclinic connections between two fixed
periodic states) has been addressed by \textsc{Koon et al.}~\cite{klmr00} in
celestial mechanics and \textsc{Kaheman et al.}~\cite{Kaheman2023} for pendulum
dynamics, using shooting, and their strategy will be used here, and we add in an
explicit computation of the asymptotic phases. The importance of the asymptotic
phase on the periodic states at infinity is brought out in the works of
\textsc{Beyn}~\cite{beyn94} and \textsc{Dieci \& Rebaza}~\cite{dr04} and we give
a new take on this by including it in the multiplicity calculation, and taking
advantage of results of \textsc{Palmer}~\cite{p00} to aid in the computation of
the stable and unstable foliations.  We add to this that the action is an
intrinsic invariant of the stable and unstable foliations, but jumps at a
surface of section.

Effective numerical strategies that have been proposed in the literature include
Lin's Method (e.g.\ \textsc{Krauskopf \& Reiss}~\cite{kr09}, \textsc{Knobloch \&
Reiss}~\cite{kr09}), boundary value problems and continuation (e.g.\
\textsc{Beyn}~\cite{beyn94}, \textsc{Pampel}~\cite{p01,p09}, \textsc{Dieci \&
Rebaza}~\cite{dr04}, \textsc{Doedel et al.}~\cite{dkvk09}), shooting (e.g.\
\textsc{Koon et al.}~\cite{klmr00}, \textsc{Barrabes et al.}~\cite{bmo13},
\textsc{Kaheman et al.}~\cite{Kaheman2023}, \textsc{Sprenger}~\cite{sbs23}), pde2path for elliptic PDEs (\textsc{Uecker et al.}~\cite{uwr14}, used for heteroclinics in \cite{kuw19}), and
the core-farfield decomposition (e.g.\ \textsc{Lloyd \& Scheel}~\cite{Lloyd2017}).
Our approach is to use a predictor-corrector strategy, with the shooting as
predictor and the core-farfield decomposition as corrector. The shooting is
effective for calculating foliations but has errors: initial data is near but
not on the periodic orbit, large Floquet multipliers can lead to growth
errors, and it is phase-space dimension limited. The corrector step is used to get improved accuracy. In addition to
accuracy, the advantage of the core-farfield decomposition is that it extends easily to
any dimension phase space, as well as to PDEs in two space dimensions.

Heteroclinic connections between periodic states are of interest in a wide range
of applications: models for localised patterns with  periodic tails (e.g.\
\textsc{Aougab et al.}~\cite{abcdssw19}, \textsc{Knobloch et al.}~\cite{kuw19}),
models of optical pulses traveling down an optical waveguide with quartic
dispersion (e.g.\ \textsc{Bandara et al.}~\cite{bgbk21,bgbk23,bgbk25}), dynamics
of coupled pendulums (e.g.\ \textsc{Kaheman et al.}~\cite{Kaheman2023}),
traveling waves with oscillatory tails in the theory of nonlinear waves (e.g.\
\textsc{Sprenger et al.}~\cite{sbs23}), and the design of spacecraft
trajectories in celestial mechanics (e.g.\ \textsc{Koon et al.}~\cite{klmr00},
\textsc{Wilczak \& Zgliczy\'nski}~\cite{wz03,wz05}, \textsc{Kirchgraber \&
Stoffer}~\cite{ks04}, \textsc{Barrab\'es et al.}~\cite{bmo13}, \textsc{Henry \&
Scheeres}~\cite{hs22}, \textsc{Spear}~\cite{s21}).

In this paper, our examples will be the Swift-Hohenberg (SH) equation with
various nonlinearities, 
the NLS equation with higher order dispersion (NLS4), and a coupled KdV system
system from the theory of water waves. In all of these systems, the steady
equation can be formulated as a Hamiltonian system. However, we will work
primarily with the equations as they arise in applications, and flag up when the
symplectic geometry informs the theory.

To illustrate our strategy, consider the simplest ODE in the above class
\begin{equation}\label{primary-ode}
u_{xxxx} + \sigma u_{xx} + V'(u) = 0 \,,
\end{equation}
for scalar-valued $u(x)$, where $V(u)$ is a given smooth nonlinear function, and
$\sigma$ is a real parameter. This ODE arises in the analysis of the steady SH
equation, time-periodic solutions of the
NLS4, and fifth-order KdV~\cite{sbs23}.  The ODE (\ref{primary-ode}) is the Euler-Lagrange equation for the
Lagrangian with density
\begin{equation}\label{L-def}
L = u_xu_{xxx} + \fr u_{xx}^2 + \fr \sigma u_x^2 - V(u)\,,
\end{equation}
and energy-type function
\begin{equation}\label{H-def}
H = u_xu_{xxx} - \fr u_{xx}^2 + \fr \sigma u_x^2 + V(u)\,.
\end{equation}
To avoid confusion with ``free energy'' (which is $L$ in this case) in
pattern formation, we will call $H$ the Hamiltonian, even though the coordinates
are not symplectic. It is a spatial Hamiltonian function as it satisfies
$dH/dx=0$, when $u$ satisfies (\ref{primary-ode}).

Suppose there exist two distinct periodic solutions of (\ref{primary-ode}) of
wavenumbers $k^-$ and $k^+$. Examples of a heteroclinic connection between
these two states satisfying (\ref{primary-ode}) are shown in Figure
\ref{fig-connection-schematic}. As $x\to-\infty$ ($+\infty$) there is a periodic
solution with a wavenumber $k^{-}$ ($k^+$), and a core region that smoothly
connects the two end states. Further details of these examples are in \S\ref{sec-SH357-review}.

However, Figure \ref{fig-connection-schematic} hides essential properties of the
heteroclinic connection.  Firstly, in order to match each periodic orbit to the
core the correct asymptotic phase has to be computed. Secondly, there is a
non-uniqueness. A pair of distinct periodic orbits on the same Hamiltonian
surface can host many heteroclinic connections. We find that generically there
are at least two, and there can be in principle be any finite number.  The
shooting strategy resolves both of these issues by computing the entire stable
and unstable foliation. Computation of the foliations using shooting is carried
out in a phase space representation of (\ref{primary-ode}). Write
(\ref{primary-ode}) as a first order system following \cite{bgbk21}, with
coordinates ${\bf u}:= (u,u_x,u_{xx},u_{xxx})$ and governing equation,
\begin{equation}\label{first-order-form}
\mathbf{u}_x = \mathbf{f}(\mathbf{u}),\qquad \mathbf{f}(\mathbf{u}) = [u_2,u_3,u_4,-\sigma u_3 - V'(u_1)]^T\,.
\end{equation}
Although not needed at this point it is useful to note that this system is
Hamiltonian with a non-canonical symplectic form
\begin{equation}\label{J-def}
  {\bf J} = \left[\begin{matrix} 0 & -\sigma & 0 & -1\\
      \sigma & 0 & 1 & 0 \\ 0 & -1 & 0 & 0 \\ 1 & 0 & 0 & 0 \end{matrix}
    \right]\,.
\end{equation}
Multiplying (\ref{first-order-form}) by ${\bf J}$ transforms the right-hand side
into $\nabla H$, with $H$ in (\ref{H-def}) expressed in terms of the
${\bf u}$ coordinates.  

The asymptotic phase on a hyperbolic periodic orbit is that value of $\theta$
for which
\begin{equation}\label{asymp-phase}
u(x) \to \uh(z+\theta,k)\,,\quad \mbox{as}\ x\to+\infty\,,
\end{equation}
when $u(x)$ is initialized on the stable manifold of a periodic state, and $\uh$
is the periodic orbit
\begin{equation}\label{u-per-orbit}
  \uh(z,k)\,,\quad \uh(z+2\pi,k)=\uh(z,k)\,,\quad z=kx\,.
\end{equation}
The union of all initial data that lands on a particular asymptotic phase is a
leaf of the stable foliation, and the stable foliation is the union over all
possible asymptotic phases. Similarly for the unstable foliation associated with
the periodic state at $-\infty$.

The shooting algorithm computes the inverse of (\ref{asymp-phase}).
For the stable foliation, a phase value $\theta^+$ is fixed and integration from
a large value of $x$ back to $x=0$ is computed, with initial data on the tangent
space of the stable manifold. Varying $\theta^+$ then creates a cylindrical manifold which
limits on a curve in the surface of section. Repeating this inverse calculation
with the hyperbolic periodic orbit at $-\infty$ creates another curve in the
surface of section. The transversal intersection of these two curves then provides a
candidate for a heteroclinic connection.  However, the situation is much more
interesting as advancing the phases $\theta^{\pm}$ from $0$ to $2\pi$ along the
periodic orbit, and integrating towards the surface of section, generates a
closed curve in the surface of section. To be precise, a closed curve without self-intersection, in the Poincar\'e section, \emph{occurs in most of our calculations}. The exception is the case where a third periodic orbit is present on the same Hamiltonian surface and is inverse-hyperbolic, distorting the curve (see Figure \ref{fig:bandara_comparison} in \S\ref{sec-numerics-I}).

A schematic of the case where we have two closed curves in the surface of section is shown in Figure \ref{fig-manifolds}(c).  The number of intersections between these two
curves gives the multiplicity of heteroclinic connections between the
$\uh^-(z+\theta^-,k^-)$ and $\uh^+(z+\theta^+,k^+)$, and values of the phase which can be used to determine the
asymptotic phases at $\pm\infty$.
\begin{figure}[ht]
\begin{center}
\includegraphics[width=\linewidth]{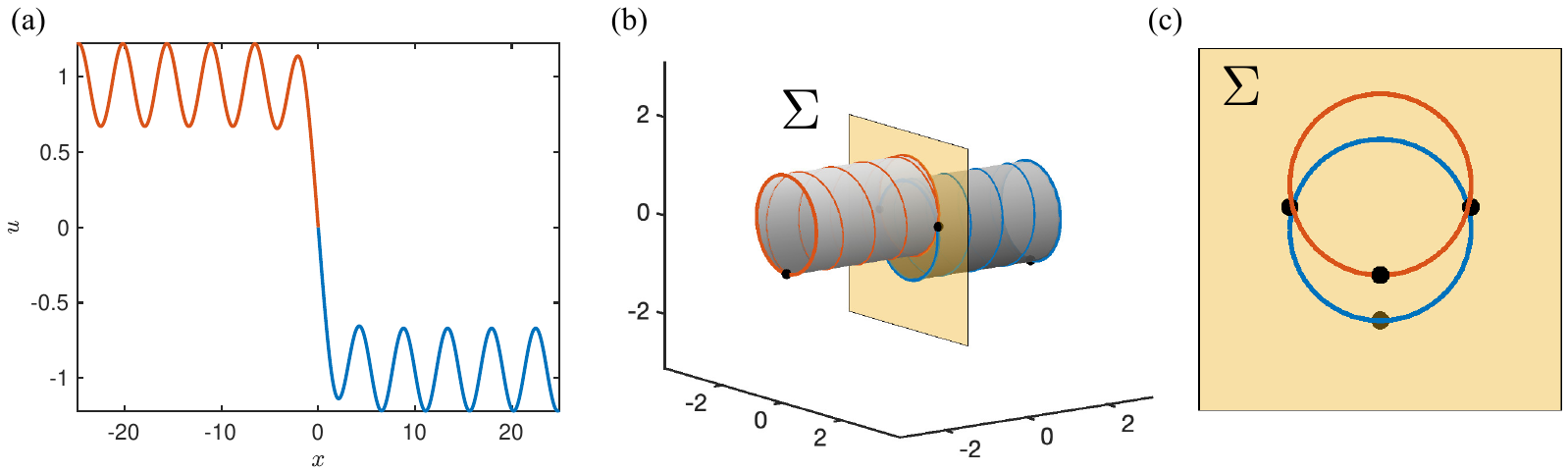}
\end{center}
\caption{Schematic of the computation of the stable and unstable foliations: panel (a) shows
  a heteroclinic connection associated with one of the intersection points in
  (c); panel (b) shows cylindrical foliations obtained by solving the
  inverse problem for the asymptotic phase; and panel (c) shows the surface of section at
  $x=0$ illustrating the intersection between the stable (blue) and unstable (red) foliations.}
\label{fig-manifolds}
\end{figure}
Figure \ref{fig-manifolds}(b) shows a schematic of the stable foliation, with a
leaf in blue on the right, and the unstable foliation, with a leaf in red on the
left. These two cylindrical foliations then intersect in a surface of section as
shown in Figure \ref{fig-manifolds}(c). Figure \ref{fig-manifolds}(a) then shows
the complete heteroclinic connection as a function of $x$.

Given a good approximation to a heteroclinic connection, with periodic
orbits $k^\pm$ and asymptotic phases $\theta^{\pm}$, it is taken as an initial
condition in the global iterative scheme of \cite{Lloyd2017}. The core-farfield
decomposition is a boundary-value solver, but differs from the strategy in BVP
solvers in \cite{beyn94,p01,p09,dr04,dkvk09} in that projection boundary
conditions are not used.  Instead zeroth-order asymptotic boundary conditions
are implemented; see \S3 and Appendix A of \cite{Lloyd2017} for an error analysis and
convergence properties. The decomposition of the solution, using
(\ref{primary-ode}) as an example, is
\[
u(x) = w(x) + \chi_+(x)\uh^+(z+\theta^+,k^+) + \chi_-(x)\uh^-(z+\theta^-,k^-)\,,
\]
with $w(x)$ the core solution, and smooth cut-off functions
\[
\chi_{\pm}(x) =1\,,\quad \pm x>d+1\qand \chi_{\pm}(x)=0\,,\quad \pm x < d\,,
\]
where $d$ is a suitably-chosen large positive constant (see Figure 9 in
\cite{Lloyd2017}).  The complete set of equations, which are then solved iteratively
using Newton's method, is given in equations (2.9)--(2.13) in \cite{Lloyd2017}.
Numerical results are reported in \S \ref{sec-SHE_PtoP} and \ref{sec-boussinesq}. 

There are two ways that the word ``codimension'' will be used here. First
there is the codimension of the intersection of the manifolds, treated as
submanifolds of the phase space and parameter space. This approach was
introduced by \textsc{Beyn}~\cite{beyn94} for heteroclinic connections between
periodic orbits and has been used explicitly or implicitly in all subsequent
numerical studies. It identifies the number of parameters in which the solutions
are generic.  This codimension will be introduced in the setup of the numerics
in \S\ref{sec-numerics-I}. The second notion of codimension used here is the order of the singularity in parameter space of a family of periodic orbits parameterized by
wavenumber.

This latter concept of codimension emerges from the normal form theory.  A
nonlinear normal form is derived in \S\ref{sec-normalform} by phase modulation
of a finite-amplitude periodic state (\ref{u-per-orbit}), starting with the
ansatz
\begin{equation}\label{uhat-modulation-0}
u(x) = \uh(z+\eps^{\rho-1}\phi,k+\eps^{\rho} q) + \eps^\mu v(z+\eps^{\rho-1}\phi,X,\eps)\,,\quad X=\eps^\nu x\,,
\end{equation}
with the constraint $\phi_X=q$ and $v$ a remainder function. The exponents
$\rho,\mu,\nu$ differ depending on the desired normal form. We claim that the nonlinear
normal form for the wavenumber modulation has the remarkably simple form
\begin{equation}\label{nf-action}
\eps^{2\nu+\rho}\mathscr{K} q_{XX} + A(k+\eps^{\rho}q) = 0 \,.
\end{equation}
The function $A(k)$ is the action evaluated on a family of periodic orbits.  It
is a relative integral invariant, and when evaluated on periodic orbits it is
\begin{equation}\label{A-lagr-formula}
A(k) = \frac{d\ }{dk} \overline{L}_k\,,
\end{equation}
where $L$ is the Lagrangian (\ref{L-def}), averaged on a family of periodic
orbits. By defining 
\[
F(q) := \overline{L}(k+\eps^\rho q)\,,
\]
the modulation equation is in the form of Newton's Law for a particle, 
\[
\eps^{2\nu+2\rho}\mathscr{K} q_{XX} = - \nabla F(q)\,,
\]
with $\mathscr{K}$ playing the role of mass.  We will have other interpretations
of $\mathscr{K}$ including strategies for computing it.

At this point the form (\ref{nf-action}) is an assertion.  We will prove it by
expanding $A(k+\eps^\rho q)$ in a Taylor series in $\eps$, then use conventional
phase modulation, based on the ansatz (\ref{uhat-modulation-0}), to construct a
normal form, order by order, and then prove that the coefficients are the Taylor
coefficients of the action. Expanding $A$ in a Taylor series in
(\ref{nf-action}),
\begin{equation}\label{nf-action-taylor}
\eps^{2\nu+\rho}\mathscr{K} q_{XX} + A(k) + A_k(k)\eps^{\rho}q + \fr A_{kk}(k)\eps^{2\rho}q^2 + \frrr A_{kkk}(k)\eps^{3\rho}q^3 + \cdots = 0 \,.
\end{equation}
We call $A_k=0$, for some $k\neq0$, along a branch of periodic orbits a
codimension-one singularity, and when $A_k=A_{kk}=0$, in a two-parameter family, we call
it a codimension-two singularity. This version of codimension points to the
importance of plotting the action as a function of $k$ along families of
periodic orbits.  A schematic of the action plotted as a function of $k$ is in
Figure \ref{fig-Hk-diagram} along with a plot of the spatial Hamiltonian of the system. 

There is a close relationship between singularities in the action, and
singularities in the Hamiltonian function.  Evaluate the Hamiltonian function
(\ref{H-def}) on a family of periodic orbits of (\ref{primary-ode}),
\begin{equation}\label{H-def-k}
H(k) := H(\uh(z,k))\,,
\end{equation}
(using the same symbol to simplify notation). We prove in \S\ref{sec-periodic-states} that $A_k=0$ whenever $H_k=0$.  Hence, the above codimension-two
singularity occurs when
\begin{equation}\label{energy-inflection}
H_k = H_{kk}=0\,,\qand H_{kkk}\neq 0 \,.
\end{equation}
The advantage of this is that the coefficients in the normal form can be also
deduced, at least qualitatively, from a Hamiltonian-wavenumber diagram.

At the codimension two point the normal form, to leading order, is
\begin{equation}\label{q-RE-0}
\eps^{2\nu+\rho}\mathscr{K} q_{XXX} + \fr {A}_{kkk}\eps^{3\rho} q^2q_X= 0 \,.
\end{equation}
We have differentiated the equation (\ref{nf-action-taylor}) as that will be the
form that emerges naturally in the conventional derivation of phase modulation
in \S\ref{sec-normalform}. The terms in (\ref{q-RE-0}) are in balance when
$\rho=\nu=1$, giving
\begin{equation}\label{q-RE}
\mathscr{K} q_{XXX} + \fr {A}_{kkk} q^2q_X= 0 \,.
\end{equation}
Taking $\mu=2$, the ansatz is
\begin{equation}\label{uhat-modulation}
u(x) = \uh(z+\phi,k+ q) + \eps^2 w(z+\phi,X,\eps)\,,\quad X=\eps x\,.
\end{equation}
It is straightforward to insert this ansatz into (\ref{primary-ode}), expand all
terms in Taylor series in $\eps$ and solve order by order.  The solvability
condition at fourth order gives a normal form like (\ref{q-RE}).  The difficult
part is then proving that the coefficients in that derivation are indeed the Taylor
coefficients of the action. The proof is sketched in Appendix \ref{app-c}, with details in the supplementary material.

A typical Hamiltonian-wavenumber diagram, which gives a starting point for
finding a codimension-two point, is shown in Figure \ref{fig-Hk-diagram}.
\begin{figure}[ht]
\begin{center}
\includegraphics[width=7.0cm]{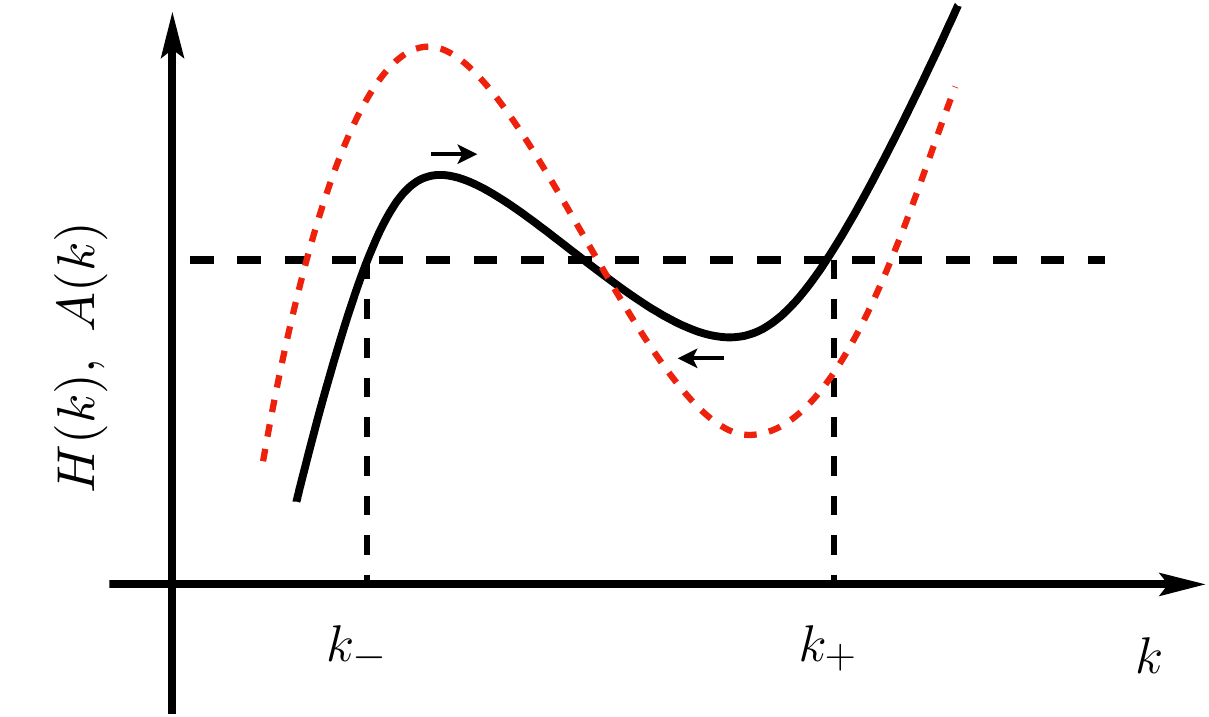}
\end{center}
\caption{Plot of the value of the Hamiltonian and the Action function versus the wavenumber along a branch of periodic states near an inflection point. The solid curve corresponds to the spatial Hamiltonian and the dashed curve to the Action. Arrows indicate the motion of the critical points as the codimension 2 point is approached. }
\label{fig-Hk-diagram}
\end{figure}
There are two critical points, one a max and one a minimum, and the
codimension-two point brings these two singularities together.  If the branch of
periodic orbits to the left of the maximum is hyperbolic, then the branch to the
right of the minimum is also hyperbolic, generically, and so can be connected,
if they are on the same Hamiltonian surface, and satisfy a transversality
condition.  Such a Hamiltonian surface can be found by drawing a horizontal
line. As the codimension-two point is approached the two hyperbolic periodic
orbits are close and explicit formulae for the heteroclinic connection is
obtained. This graphical approach is not the only way to identify pairs of periodic orbits on the same
Hamiltonian  
surface; indeed, since Hamiltonian surfaces are not necessarily connected,
periodic orbits on different components can also be paired.

The normal form is completely integrable. The constant solutions $q^{\pm}$ of
(\ref{q-RE}) represent the periodic states at infinity, and an explicit
heteroclinic exists between these two states, when $\mathscr{K}$ and $A_{kkk}$
have particular signs. A detailed analysis of the normal form in given in
\S\ref{sec-normalform}. The normal form theory extends to any conservative
system on finite dimension.  Hence the above codimension two point answers the
question of how to determine if a given system has any heteroclinic connections
at all: find a codimension two point and check the signs of $\mathscr{K}$ and $A_{kkk}$.

We sketch the origin of the action conservation law, with details in
\S\ref{sec-periodic-states}. The key to action is that it is defined on an
\emph{ensemble of solutions} $u(x,s)$, with $s$ parameterizing the ensemble.
Substituting this ensemble into the Lagrangian $L$ in (\ref{L-def}) and
differentiating gives the exact identity
\begin{equation}\label{A-claw}
L_s + u_s\Big( u_{xxxx}+\sigma u_{xx} + V'(u)\Big) = A_x\,.
\end{equation}
When $u(x,s)$ is a solution (for each $s$), the term in brackets vanishes, and
the conservation law (on ensemble space) is exact: $L_s=A_x$, with
\begin{equation}\label{A-def-2}
A(x,s) = u_{xxx}u_s + u_xu_{xxs} + \sigma u_xu_s\,.
\end{equation}
This function, and particularly its value on families of periodic orbits, will
feature prominently in the theory. Further aspects of action are developed in
\S\ref{sec-periodic-states}.

We will restrict attention to examples where the steady equation has a phase
space of dimension four.  However, much of the strategy extends to higher phase
space dimension, and even to multiple space dimension (see Concluding Remarks
\S\ref{sec-con}).

We've chosen four example PDEs to illustrate the theory.  The first equation is
the Swift-Hohenberg equation
\begin{equation}\label{SH-unsteady}
u_t + u_{xxxx} + \sigma u_{xx} + V'(u)=0\,,
\end{equation}
with polynomial $V(u)$. Two cases of SH of interest are the quadratic-cubic SH equation with
\begin{equation}\label{SH-cubic-V}
V(u) = \frac{1}{2} u^2 +\frac{1}{3}\nu u^3 - \frac{1}{4} u^4\,,
\end{equation}
with $\nu$ a specified real parameter (when $\nu=0$ it will be called the cubic SH equation),
and the SH357 equation, introduced \cite{kuw19}, where
\begin{equation}\label{SH357-V}
V(u) = \frac{1}{2}(1-\mu) u^2 + \frac{1}{4} a u^4 - \frac{1}{6}b u^6 + \frac{1}{8} u^8\,,
\end{equation}
where $a$, $b$, and $\mu$ are real parameters (we replace $\lambda$ in \cite{kuw19} with $\mu$), and $\sigma$ is taken to be $2$ for this model.
The third model is the NLS equation with
fourth-order dispersion \cite{bgbk21,bgbk23,bgbk25},
\begin{equation}\label{nls-4}
  \Psi_t = \ri\gamma |\Psi|^2\Psi - \fr \ri \beta_2 \Psi_{xx} + \ri\frac{\beta_4}{24} \Psi_{xxxx}\,,
\end{equation}
where $\Psi(x,t)$ is complex-valued, and $\gamma$, $\beta_2$, and $\beta_4$ are
real-valued parameters.  By taking $\Psi(x,t) = u(x)\re^{\ri\mu t}$, it reduces to
the steady equation (\ref{primary-ode}) for the real-valued function $u(x)$,
with
\begin{equation}\label{nls-steady-V}
  \sigma = -12\frac{\beta_2}{\beta_4} \qand
  V(u) = -12 \frac{\mu}{\beta_4}u^2  + 6 \frac{\gamma}{\beta_4}u^4\,.
\end{equation}
In the above three models, oscillatory fronts have been found previously, and we
give a new perspective and compare results. A fourth model, for which
oscillatory fronts have not been previously found, is the $ac-$Boussinesq (or coupled KdV) system
\begin{equation}\label{ac-B}
\begin{array}{rcl}
h_t + (u + hf'(u))_x +a u_{xxx} &=&0\\[2mm]
u_t + f(u)_x + gh_x + c h_{xxx} &=& 0\,,
\end{array}
\end{equation}
where $f(u)$ is a given smooth function.  The canonical case, derived from the
water-wave problem, has $f(u)=\fr u^2$. We will extend this form by including
higher order terms in $f$,
\begin{equation}\label{f-def-bouss}
f(u) = \fr u^2 + \frac{\alpha}{3} u^3, 
\end{equation}
for some specified parameter $\alpha$. The steady system relative to the moving frame, $x\mapsto x-Ct$, is
\begin{equation}\label{ac-B-steady}
\begin{array}{rcl}
c h_{xx} &=& F_h\\[2mm]
a u_{xx} &=& F_u\,,
\end{array}
\end{equation}
where $g>0$ is the gravitational constant and $a,c$ are dispersion
coefficients, and
\begin{equation}\label{ac-F-def}
F(h,u) = \iA u + \iB h - \fr gh^2 -\fr u^2 - hf(u) + C hu\,,
\end{equation}
where $\iA,\iB$ are constants of integration. The PDE (\ref{ac-B-steady}) is a model
for water waves ($\alpha=0$) and internal waves ($\alpha\neq0$), with the
coefficients $a,c$ taking specific values (cf.\ \textsc{Bona et
al.}~\cite{bcs-II}). 

The paper is outlined as follows. In Section~\ref{sec-numerics-I}, we review the geometry of stable and unstable manifolds involved in PtoP heteroclinic orbits and explain how the asymptotic phases and multiplicity are determined. As part of the review, we cover the recent numerics of \textsc{Bandara et al.}~\cite{bgbk21,bgbk23,bgbk25} for \eqref{nls-4} and reinterpret them in the context of the geometry, while also extending their result to the case where the $A\rightarrow-A$ symmetry is broken. In Section~\ref{sec-periodic-states}, we introduce the action and its link to Floquet stability analysis of spatially periodic orbits. The normal form theory is then developed in Section~\ref{sec-normalform} that derives the normal form near a Hamiltonian co-dimension 2 point. We describe, in Section~\ref{sec-SH357-review} how to find the Hamiltonian co-dimension 2 point of the periodic orbits and compute the emerging fronts for the SH357. We show in Section~\ref{sec-boussinesq} how the theory applies to the $ac-$Boussinesq equations. Finally, we conclude in Section~\ref{sec-con}. 

\section{Computing foliations in the Generalized NLS Equation}
\setcounter{equation}{0}
\label{sec-numerics-I}

In this section, we take the fourth-order NLS equation of \textsc{Bandara et al.}~\cite{bgbk21,bgbk23,bgbk25} as a starting point,
\begin{equation}\label{gnlse}
  \frac{\partial \Psi}{\partial t} = \ri \gamma|\Psi|^2\Psi -\ri\frac{\beta_2}{2}
  \frac{\partial^2\Psi}{\partial x^2} + \ri \frac{\beta_4}{24}\frac{\partial^4\Psi}{\partial x^4}\,,
\end{equation}
where $\Psi(x,t)$ is complex and the parameters $\beta_2$, $\beta_4$, and $\gamma$
are real. We have replaced coordinates $A$ and $(z,t)$ in \cite{bgbk21,bgbk23,bgbk25}
with $\Psi$ and $(t,x)$ to correspond with notation in this paper.

The aim of this section is three-fold. Firstly, to implement the algorithm for the cylindrical foliations depicted in Figure \ref{fig-manifolds} and illustrate the use of $H-k$ and $A-k$ diagrams. Secondly, to compare with previous work in \cite{bgbk21,bgbk23,bgbk25}, and thirdly to show some new computations of heteroclinic connections which break the $\Psi\mapsto -\Psi$ symmetry in (\ref{gnlse}). We do not find a codimension two point in this example; that will come later when we add a
higher order nonlinearity. 

To reduce (\ref{gnlse}) to a steady problem, time-periodic solutions of the form $\Psi(x,t) = u(x)\re^{\ri\mu t}$ are sought, reducing (\ref{gnlse}) to an ODE for $u(x)$,
\begin{equation}\label{u-ode}
  \frac{\beta_4}{24} \frac{d^4u}{dx^4} - \frac{\beta_2}{2}\frac{d^2u}{dx^2} - \mu u + \gamma u^3 = 0 \,.
\end{equation}
For the cases of interest here, we take parameter values from the papers \cite{bgbk21,bgbk23,bgbk25}: $\beta_4=-1$, $\mu=\gamma=1$. With
these values, and a scaling of $x$, the ODE (\ref{u-ode}) is reduced to the form
in (\ref{primary-ode}) with $V'(u)=u-u^3$,
  \begin{equation}\label{cubic-SH-eqn}
    u_{xxxx} + \sigma u_{xx} + u - u^3 = 0\quad\mbox{with}\quad
\sigma=\beta_2\sqrt{6} \,.
  \end{equation}
  Periodic orbits are denoted by $\uh(z,k)$ and they satisfy the ODE,
  \begin{equation}\label{per-eqn-nls}
  k^4\uh_{zzzz}+\sigma k^2\uh_{zz} + \uh - \uh^3 = 0 \,.
  \end{equation}
  We consider two cases.  Firstly, we look at solutions which have the $u\mapsto -u$ symmetry in \S\ref{subsec-ubar-connect}.  In this case the two period orbits are symmetry-related orbits.  Secondly, we look at solutions that break this symmetry: either by choosing distinct periodic orbits, in \S\ref{subsec-asym}, or by adding an explicit symmetry breaking term in the nonlinearity, in \S\ref{subsec-gnlse-symmetrybreaking}.
  
  \subsection{Connecting symmetry-related periodic orbits}
  \label{subsec-ubar-connect}

 This case gives a clear illustration
  of the cylindrical foliations of the stable and unstable manifolds, as well as
  the calculation of the asymptotic phase and the multiplicity count. 
  For the states at infinity we compute one family of periodic orbits, $\uh(z,k)$ then the other is just minus this, $-\uh(z,k)$.  
  
  For simplicity, take $\sigma=1$, and look for periodic solutions bifurcating from the constant states $u=\pm1$.  Linearizing (\ref{cubic-SH-eqn}) about these constant states and looking for normal mode
  solutions $\re^{\ri kx}$, with $\sigma=1$, gives the characteristic equation
\begin{equation}\label{char-eqn}
k^4 - k^2 -2 = 0 \,.
\end{equation}
The only positive real root is $k=\sqrt{2}$. 
There exists a pair of periodic orbits, $\uh^{\pm}(z,k^\pm)$, parameterized by
$k$, related by symmetry
\[
\uh^-(z,k^-) = -\uh^+(z,k^+):=-\uh(z,k)\,,\quad k^+=k^-:=k\,,
\]
and bifurcating from the constant solutions with $k=\sqrt{2}$. The Hamiltonian and action functions along the branch of periodic orbits are
\[
H = k^4 \uh_z\uh_{zzz} - \fr k^4\uh_{zz}^2 +\fr k^2\uh_z^2 + \fr \uh^2 - \frr \uh^4\qand A = \frac{1}{2\pi}\int_0^{2\pi} \big(\sigma k\uh_z^2 - k^3\uh_{zz}^2\big)\,\rd z\,,
\]
and they are plotted along the branch in Figure~\ref{fig-H-A-k}.
\begin{figure}[ht]
\begin{center}
\includegraphics[width=0.9\linewidth]{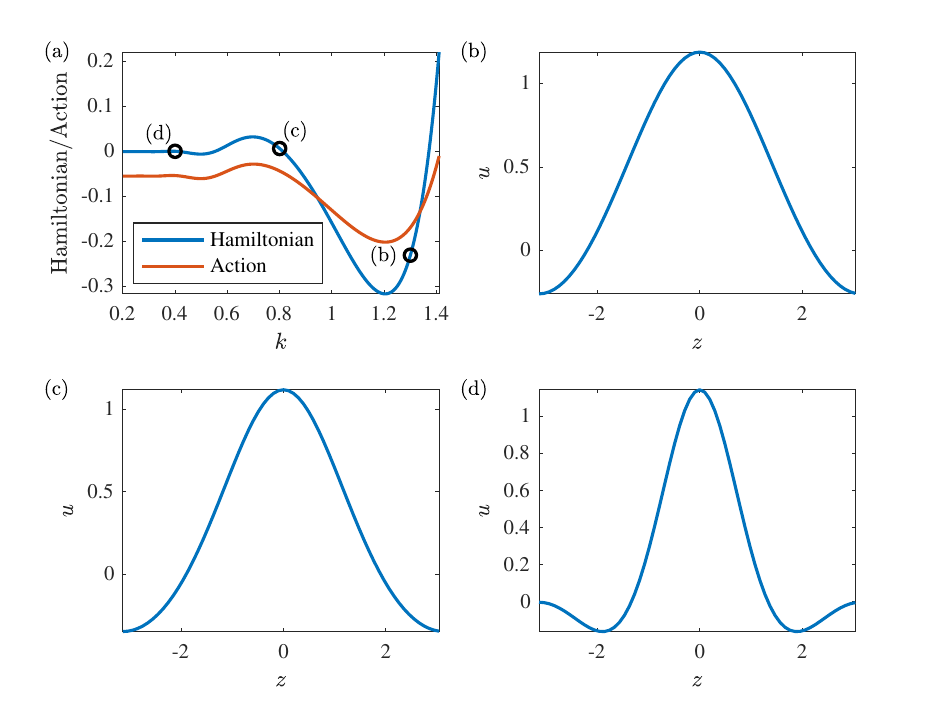}
\end{center}
\caption{A family of periodic solutions of (\ref{per-eqn-nls}) with $\sigma=1$. In panel (a) the Hamitonian and Action are plotted as functions of $k$. Although $H$ and $A$ are qualitatively different, their critical points agree. Panels (b), (c) and (d) show plots of the
  periodic solutions $\uh(z,k)$ as a function of $z$, with $-\pi < z < +\pi$, at three  $k-$points along the curves.}
\label{fig-H-A-k}
\end{figure}
The starting point in Figure \ref{fig-H-A-k} is on the right where $k=\sqrt{2}$ and $H=\frr$. The wavenumber is decreasing in Figure \ref{fig-H-A-k}(a) with increasing amplitude.

The periodic solutions along the curves are hyperbolic, at least for small amplitude (near $k=\sqrt{2},H=1/4$).  This
follows since the constant states are saddle points: the characteristic equation
(\ref{char-eqn}) has two real roots and two purely imaginary.

Using the shooting algorithm, we will show that there exist heteroclinic
connections between these two periodic orbits. For shooting we use the
first-order form (\ref{first-order-form}) which in this case is
\begin{equation}\label{first-order-numerics}
  (u_1)_x = u_2\,,\quad (u_2)_x = u_3\,,\quad (u_3)_x = u_4\,,\quad (u_4)_x = - \sigma u_3 - u_1 + u_1^3\,.
\end{equation}
To find the connections, we will compute the stable and unstable foliations and track their intersection. The definition of a foliation starts with the asymptotic phase. Denote a vector-valued periodic solution of (\ref{first-order-numerics}) by $\buhat(z,k)$.  A point in the stable manifold, ${\bf u}_0$ is said to have asymptotic phase $\theta$ if
\begin{equation}\label{asym-phase-plus}
\|\Phi({\bf u}_0)(x) - \buhat(z+\theta,k)\|\to 0 \quad\mbox{as}\ x\to+\infty\,,
\end{equation}
where $\Phi({\bf u}_0)$ is the flow of (\ref{first-order-numerics}) with initial
data ${\bf u}_0$. Each asymptotic phase defines a leaf. A leaf is the union of
all ${\bf u}_0$ with the same asymptotic phase.  The union of these leaves over
all asymptotic phases is the foliation. \textsc{Palmer}~\cite{p00} proves
the existence and uniqueness of the asymptotic phase, as well as the existence and
smoothness of the foliation. An asymptotic phase on each of the periodic
states at $\pm\infty$ is an essential part of the construction of the
heteroclinic connection.  

We follow the shooting strategy in \textsc{Koon et
al.}~\cite{klmr00} and \textsc{Kaheman et al.}~\cite{Kaheman2023} when computing
a leaf of a foliation. It is initialized with a tangent vector of the periodic orbit using the linearization of (\ref{first-order-numerics}) about the periodic orbit,
\[
\frac{\xd}{\xd x} \begin{bmatrix}
\delta u_1 \\ \delta u_2 \\ \delta u_3 \\ \delta u_4
\end{bmatrix}
=
\begin{bmatrix}
0 & 1 & 0 & 0 \\
0 & 0 & 1 & 0 \\
0 & 0 & 0 & 1 \\
-1 + 3\uh^2 & 0 & -1 & 0
\end{bmatrix}
\begin{bmatrix}
\delta u_1 \\ \delta u_2 \\ \delta u_3 \\ \delta u_4
\end{bmatrix},
\]
where $\delta u_i$ represent small perturbations about the periodic orbit
$\uh^{\pm}(z,k)$. 
We compute the stable and unstable manifolds
by integrating forward and backward in $x$, respectively, starting the integration by perturbing in the
eigendirections corresponding to the stable or unstable Floquet multipliers. For the numerical computations carried out in this section, a
typical perturbation amplitude is $\|\delta \mathbf{u}\| = 10^{-8}$. The foliation is constructed leaf by leaf by incrementing the phase along the periodic orbit.

To determine the heteroclinic connection, we compute the intersection of the
invariant manifolds on an appropriate Poincar\'e section that is transverse to the foliations, as shown schematically in Figure \ref{fig-manifolds}. For the case at hand, a natural choice for the Poincar\'e section is $u_1=0$. The intersection points single out a leaf which is a candidate connection.  The actual connection, and its asymptotic phase, is obtained by integrating from the Poincar\'e section back to the periodic orbit, along the chosen leaf.

\begin{figure}[htp]
\begin{center}
\includegraphics[scale = 0.25]{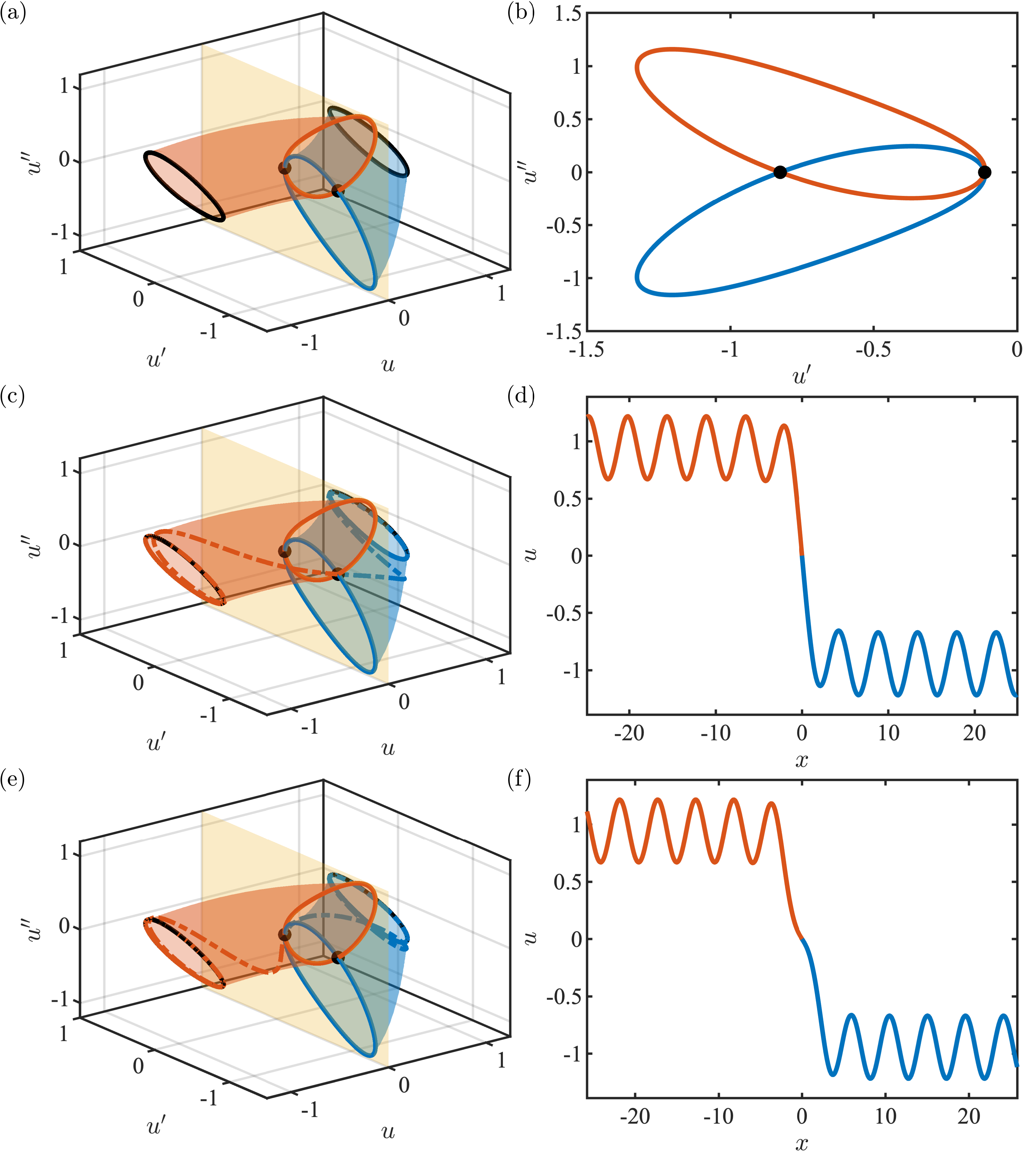}
\end{center}
\caption{A plot of the stable and unstable foliations of the periodic orbits in
\eqref{first-order-numerics}. Panel (a) shows the stable and unstable manifolds
(colored blue and red respectively) starting from the periodic orbits and ending on the (colored yellow) Poincar\'e
section. Panel (b) shows a plot of
the Poincar\'e section, with the two intersection points colored black. Panel (c) includes the manifolds along with the highlighted leaf which is the connection. That connection is shown as a function of $x$ in panel (d). Panel (e) Shows a plot of the other heteroclinic connection on the
manifolds and corresponding heteroclinic connection as a function of $x$ is in (f)}
\label{fig-pat-1}
\end{figure}

In Figure \ref{fig-pat-1}, we present results for the case where $k=1.38$.  The geometry
of the stable and unstable manifolds of the periodic orbits, the Poincar\'e
section and corresponding heteroclinic connections are shown. They are plotted in the
$(u_1,u_2,u_3)$ space, and, in plates (d) and (f), as a function of $x$. Figure \ref{fig-pat-1}(a)
shows the two periodic orbits in black, with centers at $u_1=\pm1$.  The
foliation of the unstable (stable) manifold of the periodic orbit near $-1$ is shown in red (blue).

Figure \ref{fig-pat-1}(b) shows the Poincar\'e section $u_1 = 0$ where the two
manifolds meet as distorted ellipses and it is clear that they intersect in two
points. These two points are then candidates for heteroclinic connections.
In Figure~\ref{fig-pat-1}(c)-(d), we include one of the heteroclinic connections
which smoothly connects the two periodic orbits with the other connection shown
in Figure~\ref{fig-pat-1}(e)-(f).
The connecting orbit is found by winding back leaves of the foliations to
$\pm\infty$, starting with an intersection point.  This calculation gives
the asymptotic phases which we find to be $\pm \theta$, with $\theta\approx
1.36$ (in dimensionless units), giving the asymptotic behavior
\begin{equation}\label{asymp-phase-1}
\lim_{x\to\pm\infty} \| u(x) - \uh^{\pm}(z\mp\theta,k)\| = 0\,.
\end{equation}
Due to the 
$u\rightarrow-u$ symmetry and the fact that the action is sign-invariant, the actions of the two periodic orbits at infinity are equal. 

\subsection{Connecting asymmetric periodic orbits}
\label{subsec-asym}

\begin{figure}[htp]
    \centering
    \includegraphics[scale = 0.25]{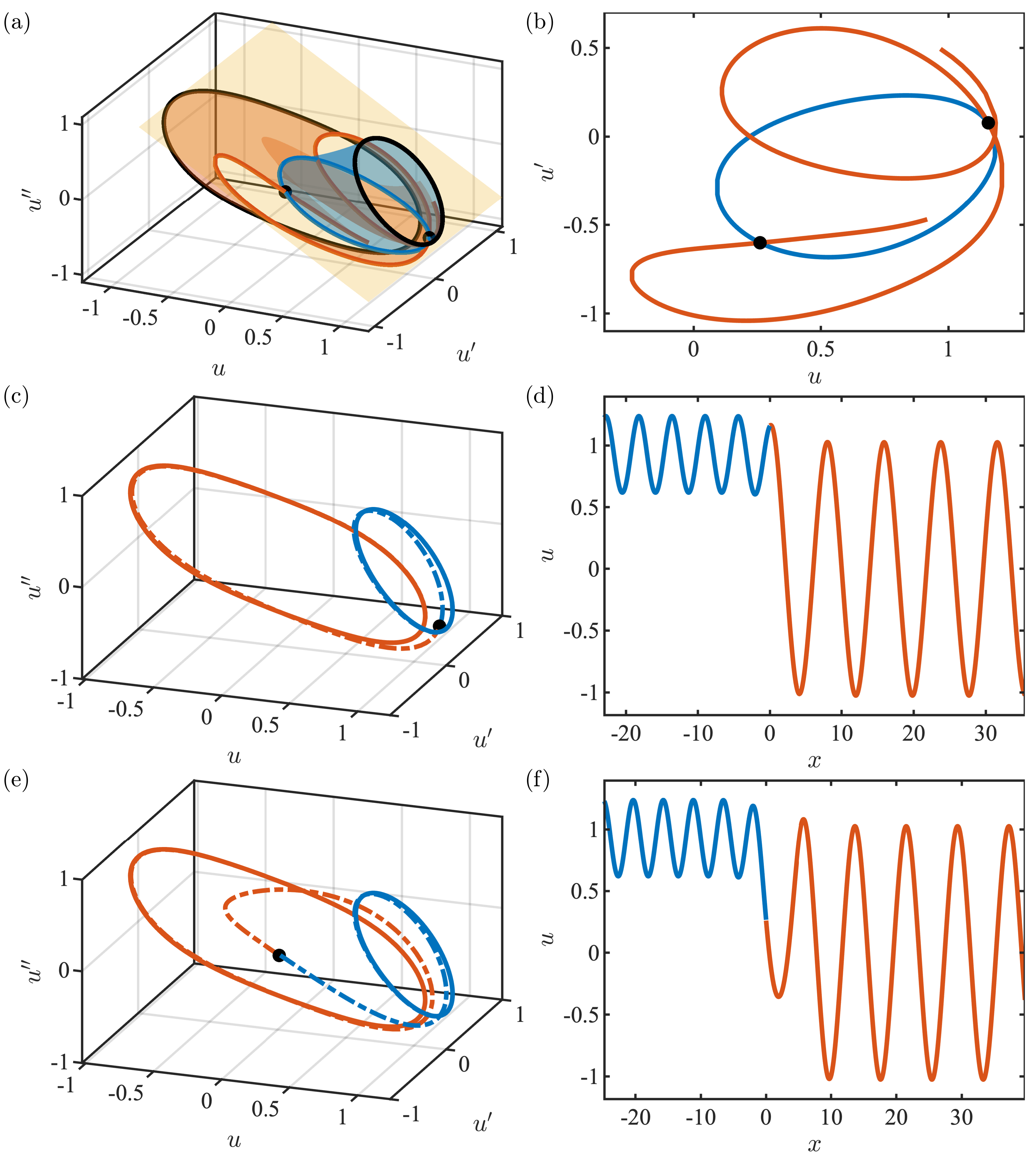}
    \caption{ Panel (a) the stable manifold is shown in red and the unstable
    manifold is shown in blue with the Poincar\'e section at $u=0$ shown in
    yellow. Panel (b) shows the stable and unstable manifolds on the Poincar\'e
    section with the two intersections identified with black dots. Panels (c) and (d)
    show one of the heteroclinic connections. Panels (e) and (f) show the other
    heteroclinic connection. \label{fig:bandara_comparison}}
\end{figure}

\begin{figure}[htp]
    \centering
    \includegraphics[scale = 0.25]{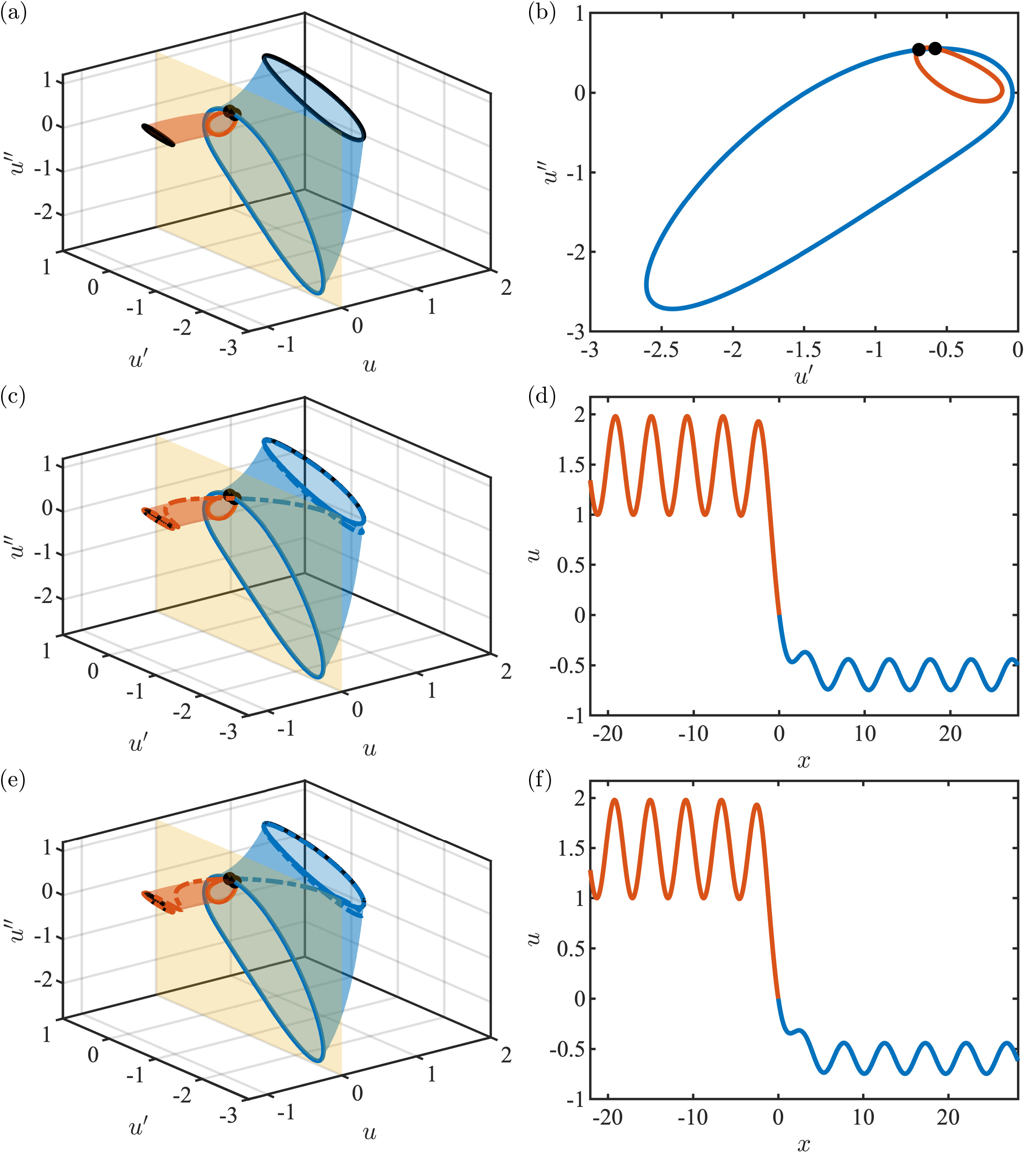}
    \caption{A plot of the stable and unstable manifolds and heteroclinic
    connections for the symmetry breaking case in (\ref{gnlse-with-quadratic-nonlinearity}). Panel (a) the stable
    manifold is shown in red and the unstable manifold is shown in blue with the
    Poincar\'e section at $u_1=0$ shown in yellow. Panel (b) shows the stable and
    unstable manifolds on the Poincar\'e section with the two intersections
    shown in black dots. Panels (c) and (d) show one of the heteroclinic
    connections and the corresponding manifolds. Panels (e) and (f) show the
    other heteroclinic connection and corresponding
    manifolds.\label{f:Pat_quad_cubic}}
\end{figure}

It is possible to connect periodic orbits with different actions in multiple
ways. One of which is to consider periodic solutions of equation
\eqref{cubic-SH-eqn} at $\pm\infty$ with different wavenumbers. Such a configuration
has been studied by \textsc{Bandara et al.}~\cite{bgbk23}.  We will recreate one of the cases in \cite{bgbk23} by direct computations of the invariant manifolds of
periodic solutions of equation \eqref{cubic-SH-eqn} with $\sigma \approx 0.97$.
The far-field hyperbolic periodic orbits have wavenumbers $k_+ \approx 1.365$
and $k_- \approx 0.799$. The wavenumbers are chosen so that the periodic orbits
lie on the level set of the Hamiltonian $H = 0$. The corresponding actions of
the periodic orbits are found to be $A_- \approx -0.18$ and $A_+
\approx -0.1268$, giving  the jump $\Delta A = 0.0532$. Figure \ref{fig:bandara_comparison} shows the computed
manifolds and the intersection with a surface of section,
$$ 1.744 u_1 + 1.929 u_3 = -0.807\,,
$$
(the values of the coefficients are approximate). This Poincar\'e section is chosen as it
effectively separates the two
periodic orbits in the phase space. The stable and unstable manifolds of the
periodic orbits are shown in red and blue, respectively. The surface of section
is shown in yellow. The intersections of the stable and unstable manifolds in
the Poincar\'e section are identified by black dots. Panels (c) and (d) show one
of the heteroclinic connections, while panels (e) and (f) show the other
heteroclinic connection. Note that the intersection of the of stable manifold
with the surface of section does not result in a closed loop. This result is
attributed to the presence of an inverse-hyperbolic periodic orbit that lies between the
the two far-field periodic orbits. Indeed, a heteroclinic connection
that is near this inverse-hyperbolic orbit is shown in Figure
\ref{fig:bandara_comparison} (e) and (f). 

\subsection{Symmetry breaking}
\label{subsec-gnlse-symmetrybreaking}

Another way to connect orbits with different actions is to include an explicit
symmetry breaking term. The additional term considered here is the quadratic
term $u_1^2$ that is added to the last equation in~\eqref{first-order-numerics}; that is, we investigate the equation 
\begin{equation}\label{gnlse-with-quadratic-nonlinearity}
u_{xxxx} + u_{xx} + u + u^2 - u^3=0
\end{equation}
In Figure~\ref{f:Pat_quad_cubic}, we illustrate the geometry of the stable and
unstable manifolds, the Poincar\'e section $u_1 = 0$, and the corresponding
heteroclinic connections that link periodic orbits with wavenumbers \( k_- =
1.503 \) and \( k_+ = 1.31 \). These particular wavenumbers are selected to
highlight the pronounced contrast in the amplitudes of the associated far-field
periodic waves. Specifically, the periodic orbit on the left, corresponding to
\( k_- \), exhibits a significantly larger amplitude than the one on the right.
This amplitude disparity is clearly reflected in the Poincar\'e section shown in
Figure~\ref{f:Pat_quad_cubic}(b), where the orbit with smaller amplitude forms a
more compact red loop, indicating a narrower projection in phase space.
The action integrals associated with these far-field periodic orbits further
quantify this difference: we compute \( A_- = -0.638 \) and \( A_+ = -0.038 \),
respectively, giving the jump $\Delta A = 0.6$.

We now consider general properties of periodic orbits, introduce the codimension two normal form, and then study two examples that exhibit the singularity.

\section{Properties of periodic orbits}
\setcounter{equation}{0}
\label{sec-periodic-states}

We will use (\ref{primary-ode}) for description of the theory, noting that it
carries over to the Boussinesq system (\ref{ac-B}) and related models, with
minor change. A periodic state of wavenumber $k$,
\begin{equation}\label{periodic-state-scalar}
\uh(z,k)\,,\quad \uh(z+2\pi,k)=\uh(z,k)\,,\quad z=kx\,,
\end{equation}
satisfies
\begin{equation}\label{ode-k}
  k^4 \uh_{zzzz} + \sigma k^2 \uh_{zz} + V'(\uh) = 0 \,.
\end{equation}
The Hamiltonian function $H$ in (\ref{H-def}) is independent of the value of
$z$, and can be used to parameterize solutions,
\begin{equation}\label{H-def-k-SH}
H = k^4\uh_z\uh_{zzz} - \fr k^4\uh_{zz}^2 + \fr \sigma k^2 \uh_z^2 + V(\uh)\,,
\end{equation}
(using the same symbol $H$ as in (\ref{H-def}) to simplify notation).

Action is an invariant on solutions that is dual to the Hamiltonian function, in
that it plays a similar role but is a relative invariant. Action is defined for
an ensemble of solutions
\begin{equation}\label{u-ensemble}
u(x,s)\,,\quad s\in(s_1,s_2)\,,
\end{equation}
where at this point $(s_1,s_2)$ is just an open interval of real numbers. The
ensemble can be thought of as parameterizing initial data, but this view is not
essential. The ensemble is assumed to depend smoothly on $s$ and satisfy
(\ref{primary-ode}) for each value of $s$,
\begin{equation}\label{primary-ode-s}
\frac{\partial^4\ }{\partial x^4}u(x,s) +
\sigma \frac{\partial^2\ }{\partial x^2}u(x,s) + V'(u(x,s)) =0\,,\quad
\mbox{for each}\ s\in(s_1,s_2)\,.
\end{equation}
As shown in the introduction action is determined by a conservation law with the
Lagrangian evaluated on the ensemble (\ref{A-claw}).  This conservation law is
confirmed by differentiating $L$ in (\ref{L-def}) with respect to $s$,
\[
\begin{array}{rcl}
L_s &=& \partial_{s}\left(u_xu_{xxx} + \fr u_{xx}^2 + \fr\sigma u_x^2 - V(u)\right)\\[2mm]
&=& u_{xs}u_{xxx} + u_xu_{xxxs}
 +  u_{xx}u_{xxs} + \sigma u_xu_{xs} - V'(u)u_s\\[2mm]
&=& \big(u_{xxx}u_s + u_xu_{xxs} +\sigma u_xu_s\big)_x 
- u_s \big( u_{xxxx} +\sigma u_{xx} +V'(u)\big)\,,
\end{array}
\]
and noting that the second term on the right-hand side vanishes on solutions.
The first term then gives the action
\begin{equation}\label{action-def}
A(x,s) = u_{xxx}u_s + u_xu_{xxs} + \sigma u_xu_s\,.
\end{equation}
The fact that $A$ is an integral invariant follows by taking $s\in S^1$ and
integrating the formula (\ref{A-claw}) over a period, giving
\begin{equation}\label{action-cons}
  \frac{d\ }{dx}\oint_{S^1} \left(
  u_{xxx}u_s + u_xu_{xxs} + \sigma u_xu_s\right)\,\rd s = 0 \,.
\end{equation}
Figure \ref{fig-action-cylinder} shows a schematic of action conservation. A
closed curve of initial data has a value of action, and propagating this closed
curve with the flow of the differential equation creates a cylinder, and the
value of action is the same in each $x-$slice.
\begin{figure}[ht]
\begin{center}
\includegraphics[width=8.0cm]{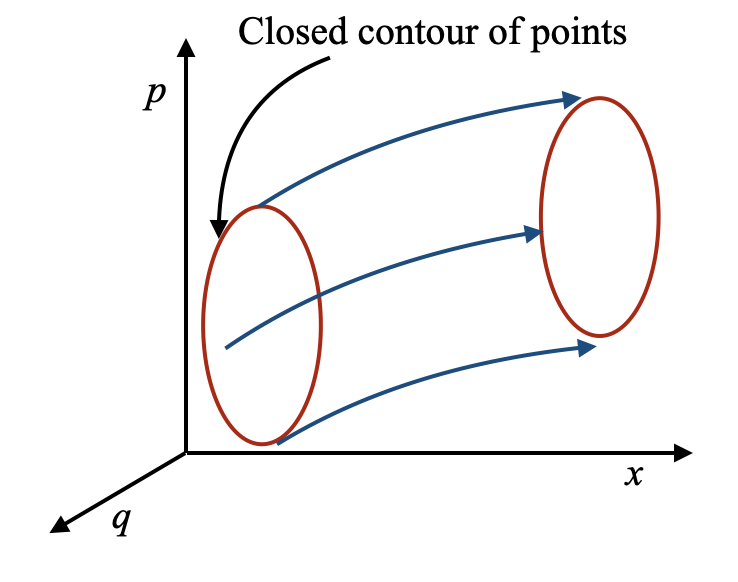}
\end{center}
\caption{Illustration of the Poincar\'e-Cartan theorem, showing how a closed curve of initial data in the phase space carries a constant value of the action under the flow. $(q,p)$ represent symplectic coordinates (see Appendix \ref{sec-symplectic}). 
}
\label{fig-action-cylinder}
\end{figure}
When transformed to symplectic coordinates, this is an example of the Poincar\'e-Cartan theorem (see Appendix \ref{sec-symplectic}). It also explains why the value of the action integral is constant along the stable and unstable foliations, although the value is not defined on distinct leaves of a foliation. This property is in contrast to the Hamiltonian which is constant along each leaf.

A periodic orbit is a special case of an ensemble of points. We can evaluate $A$
in (\ref{action-def}) on a periodic orbit by taking $s=z$ and $x=z/k$, giving
\begin{equation}\label{A-per}
  A(k) = \frac{1}{2\pi}\int_0^{2\pi} \big( \sigma k \uh_z^2 - 2k^3 \uh_{zz}^2\big)\,\rd z\,,
\end{equation}
(using the symbol $A$ again to simplify notation). Since the value of
action on a closed curve is an invariant (\ref{action-cons}), the value of
$A(\uh)$ on the periodic orbits at $\pm\infty$ will be propagated on the stable
and unstable foliations, when the curve in the integral (\ref{action-cons}) is
taken to be transverse to the leaves. That is, each of the two cylinders in
Figure \ref{fig-manifolds} carries its own action invariant, determined by the
value on the periodic orbit at infinity.

The formula (\ref{A-lagr-formula}) is proved as follows.  The averaged
Lagrangian is
\begin{equation}\label{L-avg}
\overline{L} = \frac{1}{2\pi} \int_0^{2\pi} \left( k^4\uh_z\uh_{zzz}+\fr k^4 \uh_{zz}^2 + \fr\sigma k^2\uh_z^2 - V(\uh)\right)\,\rd z\,.
\end{equation}
Integrating by parts, and differentiating with respect to $k$ gives $\overline{L}_k=A(k)$ for $A(k)$ in
(\ref{A-per}).

At points along the periodic orbits where $A_k\neq0$ one pair of Floquet
multipliers is at $+1$ and the other pair is either elliptic, hyperbolic, or
coalesce at $-1$.  Figure \ref{fig-hyperbolic} shows the two generic
configurations. In all the results in this paper, the two periodic orbits are hyperbolic.
\begin{figure}[ht]
\begin{center}
\includegraphics[width=8.0cm]{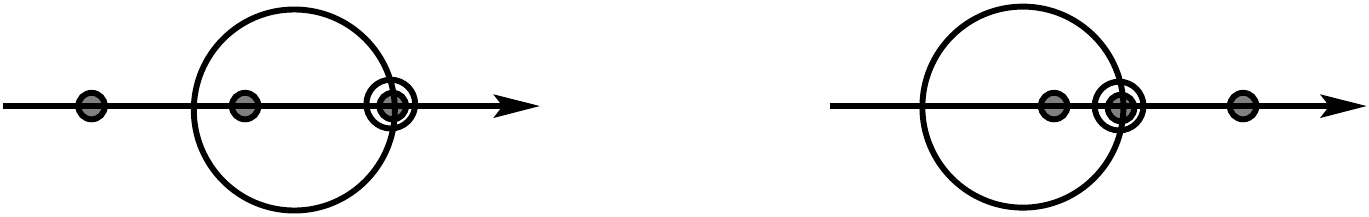}
\end{center}
\caption{The two configurations of hyperbolic Floquet multipliers for periodic
  solutions of (\ref{primary-ode}). The configuration on the left is called
  ``inverse-hyperbolic'' and on the right is ``hyperbolic''.}
\label{fig-hyperbolic}
\end{figure}

The action values of the periodic orbits at $\pm\infty$ are not equal in
general, and so there is a jump in the action when the two cylinders meet in the
surface of section.  This jump gives an intrinsic invariant of the connection
\begin{equation}\label{delta-A-def}
\Delta A := A(k^+)-A(k^-)\,.
\end{equation}
This property is new as far as we are aware.  Examples with numerical values of this jump are in \S\ref{subsec-ubar-connect} and \S\ref{subsec-gnlse-symmetrybreaking}. On the other hand, the
significance of this jump property is not apparent. The jump is zero when the
two periodic orbits match: a homoclinic to the same periodic orbit, or when two
periodic orbits are separated in phase space but are related by symmetry (an
example is in \S\ref{subsec-ubar-connect}). The action appears to be the same on each cylinder when the system is integrable, but we do not have a proof of this. The jump is intrinsic in the sense that
it is invariant under symplectic change of coordinates; this follows by expressing the action in symplectic coordinates in which case the density is ${\bf p}\cdot d{\bf q}$ (see Appendix \ref{sec-symplectic}).

It was pointed out in the introduction that action and the Hamiltonian function have critical points,
$H_k=A_k=0$, at the same value of $k$, along a branch of periodic orbits. We now prove this and give an illustration.

\begin{lem}\label{lemma-Hk} For the equation (\ref{primary-ode}) with Hamiltonian $H$ in (\ref{H-def}) and action (\ref{A-per}) evaluated on a family of periodic orbits with $k\neq0$,
    \begin{equation}\label{HkAk}
H_k = k A_k\,.
\end{equation}
 \end{lem}
 
\begin{proof}
    First, since the value of $H$ is independent of $z$, we have
\begin{equation}\label{Hk-per}
H = \frac{1}{2\pi}\int_0^{2\pi} H\,\rd z =
\frac{1}{2\pi}\int_0^{2\pi} \left( -\frac{3}{2}k^4\uh_{zz}^2+
\fr\sigma k^2 \uh_{z}^2 + V(\uh)\right)\,\rd z \,,
\end{equation}
after integration by parts. Differentiate this expression with respect to $k$,
\[
H_k =
\frac{1}{2\pi}\int_0^{2\pi} \left( - 6k^3\uh_{zz}^2
-3k^4\uh_{zz}\uh_{zzk} +\sigma k \uh_{z}^2
+\sigma k^2 \uh_{z}\uh_{zk}+ V'(\uh)\uh_k\right)\,\rd z \,.
\]
Rearrange, using the fact that $\uh$ is a solution
\[
H_k =
\frac{1}{2\pi}\int_0^{2\pi} \left(\sigma k \uh_z^2
+2\sigma k^2 \uh_z\uh_{zk} - 6k^3\uh_{zz}^2 -4k^4\uh_{zz}\uh_{zzk}
+ \Big[k^4 \uh_{zz}\uh_{zzk}- \sigma k \uh_z\uh_{zk}
+ V'(\uh)\uh_k\Big]\right)\,\rd z \,.
\]
The term inside the square brackets vanishes on solutions (after integration by
parts), leaving
\[
H_k =
\frac{1}{2\pi}\int_0^{2\pi} \left(\sigma k \uh_z^2
+2\sigma k^2 \uh_z\uh_{zk} - 6k^3\uh_{zz}^2 -4k^4\uh_{zz}\uh_{zzk}
\right)\,\rd z\,.
\]
Now differentiate $A(k)$ in (\ref{A-per}),
\[
A_k = \frac{1}{2\pi}\int_0^{2\pi}\left(
\sigma \uh_z^2 + 2\sigma k \uh_z\uh_{zk} -6k^2\uh_{zz}^2 - 4k^3\uh_{zz}\uh_{zzk}\right)\,\rd z\,.
\]
Comparing the two proves the result. That is, when $k\neq0$, we have
$kA(k)=H(k)$.
\end{proof}

\begin{rmk} Although (\ref{HkAk}) was proved by direct calculation, it is in fact a
consequence of the symplectic structure, and we will show that after the
symplectic structure is introduced in Appendix \ref{sec-symplectic}.
\end{rmk}
An example of the lemma is shown in
Figure \ref{fig-H-A-k}, where actual calculations of periodic solutions of the cubic SH equation are used.

The equivalence of inflection points in (\ref{HkAk}) does not carry over to the second
derivative since
\[
H_{kk} = A_k + kA_{kk}\,.
\]
However, when $A_k=0$, then $H_{kk}$ and $A_{kk}$ are zero at the same value of
$k$, when $k\neq 0$. This property will be useful in the normal form theory in
\S\ref{sec-normalform}.

\begin{rmk}
One of the main observations here is that parameterization of the periodic
orbits via the Hamiltonian function (which is not uncommon) has the same key
properties as parameterization by action (which is unheard of).  Action is much
easier to work with since it is always a quadratic function.  Indeed, action in
(\ref{A-per}) is a linear combination of the (square of the) $H^1(S^1)$ and
$H^2(S^1)$ Sobolev norms.
\end{rmk}

\subsection{Action and algebraic multiplicity}
\label{subsec-floquet}

The linearization about a periodic state brings in Floquet theory, but it is a
little more interesting in that the spectral problem is nonlinear in the
parameter.
 The most interesting feature is how action feeds
into the Jordan chain theory.

Let $u=\uh + \widetilde v(z,x)$, substitute into (\ref{primary-ode}), and
linearize about $\uh$,
\begin{equation}\label{sh-eqn-zxs}
\left(\frac{\partial\ }{\partial x} + k \frac{\partial\ }{\partial z}\right)^4\widetilde v
 + \sigma \left(\frac{\partial\ }{\partial x} + k \frac{\partial\ }{\partial z}\right)^2\widetilde v +V'(\uh)\widetilde v = 0 \,.
\end{equation}
Introduce a spectral ansatz in $x$, $\widetilde v(z,x) = \re^{\lambda x}
v(z,\lambda)$, giving the nonlinear in the parameter spectral problem
\begin{equation}\label{L-lambda-def}
{\bf L}(\lambda)v := 
\left(\lambda + k \frac{\partial\ }{\partial z}\right)^4v
 + \sigma \left(\lambda + k \frac{\partial\ }{\partial z}\right)^2v +V''(\uh)v\,.
\end{equation}
The tangent vector to the solution is an eigenvector with eigenvalue zero,
\begin{equation}\label{xi1-def}
  {\bf L}(0)\xi_1 = 0\,,\quad \mbox{with}\quad \xi_1 = \uh_z\,.
\end{equation}
  We assume the the geometric multiplicity is one. We define algebraic
  multiplicity for a nonlinear in the parameter eigenvalue problem following
  \textsc{Magnus}~\cite{magnus}, \textsc{Rabier}~\cite{rabier}, and
  \textsc{L\'opez-G\'omez \& Mora-Corral}~\cite{LGMG-book}.  
We have the following result from the above references, stated for the case
needed here (geometric multiplicity one and algebraic multiplicity four).

\begin{prop} \cite{LGMG-book,magnus,rabier} {\it Suppose an eigenvalue $\lambda_*$  of ${\bf
L}(\lambda)$ has geometric multiplicity one and algebraic multiplicity four.
Then there exists a Jordan chain $\xi_1,\ldots,\xi_4$, with $\xi_1$ the
geometric eigenvector satisfying}
\[
{\bf L}(\lambda_*)\xi_1 =0\,,
\]
{\it and three generalized eigenvectors satisfying}
\[
\begin{array}{rcl}
{\bf L}(\lambda_*)\xi_2 &=& - {\bf L}'(\lambda_*)\xi_1\\[2mm]
{\bf L}(\lambda_*)\xi_3 &=& - {\bf L}'(\lambda_*)\xi_2 - \fr {\bf L}''(\lambda_*)\xi_1\\[2mm]
{\bf L}(\lambda_*)\xi_4 &=& - {\bf L}'(\lambda_*)\xi_3 - \fr {\bf L}''(\lambda_*)\xi_2 - \frac{1}{3!}{\bf L}'''(\lambda_*) \xi_1\,.
\end{array}
\]
{\it Moreover, when ${\bf L}(\lambda_*)$ is self-adjoint, the chain ends at four
when}
\begin{equation}\label{jordan-four}
\left\langle\!\left\langle \xi_1(z),-{\bf L}'(\lambda_*)\xi_4 - \fr {\bf L}''(\lambda_*)\xi_3 - \frac{1}{3!}{\bf L}'''(\lambda_*) \xi_2
- \frac{1}{4!}{\bf L}''''(\lambda_*) \xi_1\right\rangle\!\right\rangle\neq 0\,,
\end{equation}
{\it where the inner product is for periodic scalar-valued functions,}
\begin{equation}\label{ip-def}
\lth a,b\rth := \frac{1}{2\pi}\int_0^{2\pi} a(z)b(z)\,\rd z \,.
\end{equation}
\end{prop}

\noindent Applying this result to (\ref{L-lambda-def}) with $\lambda_*=0$ and
the geometric eigenvector $\xi_1$ satisfying (\ref{xi1-def}), the chain
$(\xi_1,\xi_2,\xi_3,\xi_4)$ is defined by
\begin{equation}\label{chain-4}
\begin{array}{rcl}
{\bf L}(0)\xi_1 &=&\displaystyle 0\\[3mm]
{\bf L}(0)\xi_2 &=&\displaystyle -\left( 4k^3 \frac{\partial^3\ }{\partial z^3}\xi_1
 + 2\sigma k \frac{\partial\ }{\partial z}\xi_1\right)\\[3mm]
{\bf L}(0)\xi_3 &=&\displaystyle -\left( 4k^3 \frac{\partial^3\ }{\partial z^3}\xi_2
 + 2\sigma k \frac{\partial\ }{\partial z}\xi_2\right) -
6k^2\frac{\partial^2\ }{\partial z^2}\xi_1 - \sigma\xi_1\\[3mm]
{\bf L}(0)\xi_4 &=&\displaystyle -\left( 4k^3 \frac{\partial^3\ }{\partial z^3}\xi_3
 + 2\sigma k \frac{\partial\ }{\partial z}\xi_3\right) -
\left( 6k^2\frac{\partial^2\ }{\partial z^2}\xi_2 + \sigma\xi_2\right)
 - 4k \frac{\partial\ }{\partial z}\xi_1\,.
\end{array}
\end{equation}
For the spectral problem (\ref{L-lambda-def}), the algebraic multiplicity is automatically two, since the $\xi_2$ equation is
solvable,
\[
\left\langle\!\left\langle\xi_1,-\left( 4k^3 \frac{\partial^3\ }{\partial z^3}\xi_1
 + 2\sigma k \frac{\partial\ }{\partial z}\xi_1\right)
\right\rangle\!\right\rangle =0\,.
\]
That this expression is zero can be verified using integration by parts.  There
is also an explicit expression for $\xi_2$, which is obtained by differentiating
the governing equation for $\uh$ in (\ref{ode-k}) with respect to $k$,
\begin{equation}\label{ode-k-SH}
  k^4 (\uh_k)_{zzzz} + \sigma k^2 (\uh_k)_{zz} + V''(\uh)\uh_k
  + 4k^3 \uh_{zzzz} + 2\sigma k \uh_{zz}  = 0 \,.
\end{equation}
This is the equation for $\xi_2$, giving
\[
\xi_2 = \uh_k + C \uh_z\,,
\]
where $C$ is an arbitrary constant.  The most interesting result is the
condition for algebraic multiplicity three, which brings in the action.
\vspace{.15cm}

\begin{prop} {\it Suppose the above spectral problem has
$\lambda=0$ as an eigenvalue of geometric multiplicity one. The algebraic
multiplicity is automatically two. The algebraic multiplicity is three if and
only if $A_k=0$. Moreover when it is three it is automatically four.}
\end{prop}

\begin{proof}
Algebraic multiplicity two is proved above. Algebraic
multiplicity three follows if the $\xi_3$ equation is solvable; that is
\[
\left\langle\!\left\langle\xi_1,-\left( 4k^3 \frac{\partial^3\ }{\partial z^3}\xi_2
 + 2\sigma k \frac{\partial\ }{\partial z}\xi_2\right) -
6k^2\frac{\partial^2\ }{\partial z^2}\xi_1 - \sigma\xi_1\right\rangle\!\right\rangle =0\,.
\]
Integrating by parts and using $\xi_1=\uh_z$ and $\xi_2=\uh_k$, gives
\begin{equation}\label{solv-3}
\left\langle\!\left\langle\uh_z,-\left( 4k^3 \frac{\partial^3\ }{\partial z^3}\uh_k
 + 2\sigma k \frac{\partial\ }{\partial z}\uh_k\right) -
6k^2\frac{\partial^2\ }{\partial z^2}\uh_z - \sigma\uh_z\right\rangle\!\right\rangle \,.
\end{equation}
Noting that 
\[
A_k = \frac{1}{2\pi}\int_0^{2\pi}\left(
\sigma \uh_z^2 + 2\sigma k \uh_z\uh_{zk} - 6k^2\uh_{zz}^2 -
4k^3\uh_{zz}\uh_{zzk}\right)\,\rd z\,,
\]
for $A(k)$ defined in (\ref{A-per}) shows that the integral in (\ref{solv-3}) is
$-A_k$. This proves that algebraic multiplicity three is in one-to-one
correspondence with $A_k=0$. That the algebraic multiplicity is automatically
four follows from solvability of the $\xi_4$ equation, which can be verified by
direct calculation.  That completes the proof. 
\end{proof}

Although we have used the nonlinear in the parameter theory, the above Jordan
chain and multiplicity results are a consequence of the underpinning symplectic structure and comments on this are in Appendix \ref{sec-symplectic}.

The algebraic multiplicity of the above eigenvalue problem will terminate at
four if the equation for $\xi_5$ is not solvable,
\begin{equation}\label{xi5-eqn}
{\bf L}(0)\xi_5 = -\left( 4k^3 \frac{\partial^3\ }{\partial z^3}\xi_4
 + 2\sigma k \frac{\partial\ }{\partial z}\xi_4\right) -
\left( 6k^2\frac{\partial^2\ }{\partial z^2}\xi_3 + \sigma\xi_3\right)
 - 4k \frac{\partial\ }{\partial z}\xi_2 - \xi_1\,.
\end{equation}
Checking solvability,
\[
\lthb \xi_1,\left( 4k^3 \frac{\partial^3\ }{\partial z^3}\xi_4
 + 2\sigma k \frac{\partial\ }{\partial z}\xi_4\right) +
\left( 6k^2\frac{\partial^2\ }{\partial z^2}\xi_3 + \sigma\xi_3\right)
 + 4k \frac{\partial\ }{\partial z}\xi_2 + \xi_1\rthb \neq 0 \,.
\]
Integrate by parts and use periodicity
\[
-4k^3\lthb \frac{\partial^3\ }{\partial z^3}\xi_1,\xi_4\rthb
-2\sigma k \lthb \frac{\partial\ }{\partial z}\xi_1,\xi_4\rthb
 + 6k^2\lthb \frac{\partial^2\ }{\partial z^2}\xi_1,\xi_3\rthb
+ \sigma \lthb\xi_1,\xi_3\rthb 
 - 4k \lthb\frac{\partial\ }{\partial z}\xi_1,\xi_2\rthb
 + \lthb\xi_1,\xi_1\rthb \neq 0 \,.
\]
Now, note that $\xi_1=\uh$ and substitute
\begin{equation}\label{xi5-Kappa}
-4k^3\lthb \uh_{zzzz},\xi_4\rthb
-2\sigma k \lthb \uh_{zz},\xi_4\rthb
 + 6k^2\lthb \uh_{zzz} , \xi_3\rthb
+ \sigma \lthb \uh_z,\xi_3\rthb 
 - 4k \lthb \uh_{zz},\xi_2\rthb
 + \lthb\uh_z,\uh_z\rthb \neq 0 \,.
\end{equation}
This condition should be satisfied for termination of the chain at four. We will
find that this expression is precisely the coefficient $\mathscr{K}$ in the
normal form in (\ref{q-RE}) and (\ref{ch-eqn}).

\begin{rmk} It is well known in the  pattern formation literature and
in the celestial mechanics literature, for example, that $H_k=0$ signals a
turning point in a branch of periodic solutions.  What is surprising here is
that the above direct proof of this is in terms of $A_k=0$, with the association
with $H_k=0$ following indirectly from (\ref{HkAk}).
\end{rmk}


\section{Nonlinear normal form theory}
\setcounter{equation}{0}
\label{sec-normalform}

In this section, we discuss the form and features of the nonlinear normal form
that describes the wavenumber modulation near the codimension two point, $A_k=A_{kk}=0$, using phase modulation~\cite{Hoyle2006}. We then impose the symplectic structure on the phase modulation to prove that the key coefficient in the normal form is related to the action.

The emergence of the normal form starts with the ansatz  
\begin{equation}\label{modulation-ansatz}
u(x) = \uh(z+\phi,k_0+\eps q) + \eps^2 w(z+\phi,X,\eps)\,,
\end{equation}
where $k_0$ is the wavenumber of the unperturbed wave and
\begin{equation}\label{q-def}
X=\eps x\,,\quad q=\phi_X\,,\quad \eps^2w = \eps^2w_2 + \eps^3 w_3 +
\eps^4 w_4 + \cdots\,.
\end{equation}
The ansatz (\ref{modulation-ansatz}) is substituted into (\ref{primary-ode}),
everything is expanded in Taylor series in $\eps$, and solved order by order.
The full details of this procedure can be found in Appendix \ref{mod-reduction},
but we note here the central role of action.  The conditions $A_k(k_0) = 0$ and $A_{kk}(k_0) = 0$ arise as
solvability conditions within the analysis for the $q_X$ and $qq_X$ terms, in the normal form,
respectively. The former of these is the primary singularity in the
modulation approach, and results in the emergence of a third order derivative in the normal form.
The second of these brings in the cubic nonlinearity as a natural
consequence, taking the role of the dominant nonlinearity in the normal
form. As confirmed by a direct asymptotic analysis in Appendix \ref{mod-reduction}, we obtain the normal form at the critical point:
\begin{equation}\label{nf-1}
  \frac{1}{2}A_{kkk}(k_0) q^2q_X + \mathscr{K} \,q_{XXX}= 0 \,.
\end{equation}
The coefficient $\mathscr{K}$ of the spatial derivative can be characterized various ways.  It emerges as a solvability condition in the Jordan chain theory (e.g. Equation (\ref{xi5-Kappa})), and can also be characterized using symplectic Jordan chain theory.  Here we will relate it to the Bloch spectrum of the basic wavetrain $\mu(\nu;k)$ via
\begin{equation}\label{mu-bloch}
\mathscr{K} = \frac{1}{24} \mu''''(0;k_0)
\end{equation}
where primes denote derivatives with respect to $\nu$. This formula is proved in Appendix \ref{app:bloch}. The connection between dispersion and the Bloch exponent was first
identified in dissipative pattern forming systems \cite{doelman2009dynamics} and later
generalized to Hamiltonian systems, and related to Krein signature, in \cite{ratliff2021}. The result is a normal
form that can be constructed entirely using properties of the original wave -
its action and its Bloch spectrum - which can be computed independently of the
asymptotic analysis.

The analysis of this paper is interested in behavior in the vicinity of this
codimension 2 point, where the heteroclinic connections develop further, and
thus we must introduce an unfolding of the codimension 2 point. Formally, the
resulting asymptotic balance requires that
\[
A_k(k_0) = \mathcal{O}(\eps^2)\,, \quad A_{kk}(k_0) = \mathcal{O}(\eps)\,,
\]
and in the sequel we will assume the asymptotic constants associated with these
orderings is $\alpha$ and $\beta$ respectively. Examples of these asymptotic
constants and how they can be obtained is given in
\S\ref{subsec-nf-solutions}. In such cases, the generic unfolding of the
codimension 2 point is as follows:
\begin{equation}\label{q-RE-unfolding-2}
\frac{1}{2} A_{kkk}(k_0)\, q^2q_X + \alpha\, q_X + \beta\, qq_X +\mathscr{K} \, q_{XXX} = 0\,.
\end{equation}
The details of the derivation of (\ref{nf-1}) are given in Appendix
\ref{mod-reduction}.  It is the unfolded version (\ref{q-RE-unfolding-2}) that
we will explore in this paper, as it has explicit heteroclinic connections, noting that the true
codimension 2 point follows by setting $\alpha = \beta = 0$ .

\subsection{Heteroclinic connections via the normal form}
\label{subsec-nf-solutions}

We now explore (\ref{q-RE-unfolding-2}) to determine the criterion for
heteroclinic connections and their mathematical form. We start by integrating
(\ref{q-RE-unfolding-2}) once, resulting in the ODE
\begin{equation}\label{ODE-form}
\mathscr{K} q_{XX} + \alpha q + \frac{1}{2}\beta q^2 + sq^3 + I =0\,, \quad s = \frac{1}{6}A_{kkk}
\end{equation}
where $I$ is the a constant of integration. We can see that $I$ plays the role
of the action of this orbit evaluated on the original unperturbed wave. This can
best be seen comparing the Taylor series of the action,
\[
A(k)\approx A(k_0)+A_{k}(k_0)(k-k_0)+\frac{1}{2}A_{kk}(k_0)(k-k_0)^2+\frac{1}{6}A_{kkk}(k_0)(k-k_0)^3
\]
to the polynomial terms in (\ref{ODE-form}), recalling the definitions of
$\alpha$ and $\beta$, and noting that $k-k_0\approx \eps q$. It is useful to our
discussion to notice that (\ref{ODE-form}) is of the form of a kinetic-potential
Hamiltonian system, with Hamiltonian
\[
H_{NF}(q,q_X) = \frac{1}{2}\mathscr{K} q_X^2+I q + \frac{1}{2}\alpha q^2 + \frac{1}{6}\beta q^3 + \frac{s}{4} q^4 \equiv \frac{1}{2}\mathscr{K} q_X^2+\mathcal{H}(q)\,,
\]
where crucially we note the potential function $\mathcal{H}$ is quartic. The
remainder of our analysis will be to determine the conditions on this quartic
potential that permit heteroclinic connections in the wavenumber $q$ and thus a
PtoP connection of the original underlying wave.

Constant solutions of the normal form (\ref{ODE-form}) correspond to periodic
orbits of (\ref{primary-ode}) and are those with fixed action, satisfying
\begin{equation}\label{I-eqn}
I + \alpha q + \frac{1}{2}\beta q^2 + s q^3 =0\,.
\end{equation}
This is to say that they minimize the potential function $\mathcal{H}$, as the
above polynomial corresponds to the condition $\mathcal{H}'(q) = 0$. Two solutions of this
equation are called conjugate periodic orbits if they both satisfy (\ref{I-eqn})
for the same value of $I$ and $H_{NF}$. 

\begin{rmk} As equilibria of the normal form conjugate orbits have the same action and Hamiltonian, since the normal form is integrable.  However, when lifted to the phase space, the representative periodic orbits will not, in most cases, have the same action, although they will have the same value of the Hamiltonian.
\end{rmk}

Another property of conjugate periodic orbits is that they share the same potential energy $\mathcal{H}(q)$. Thus, our criterion for
conjugate states is entirely determined by the shape of the potential energy -
it must possess two extrema that occur at the same energy level. As such, our
condition for conjugate states reduces to requiring that $\mathcal{H}(q)$ is of
the form
\[
\mathcal{H}(q) = \frac{s}{4}(q-Q_1)^2(q-Q_2)^2
\]
for two real numbers $Q_1\neq Q_2$. We now seek to find the general form of the conjugate states in terms of the
coefficients of the normal form. Denote the two solutions of this problem by
$Q_1$ and $Q_2$ with $Q_1\neq Q_2$. By definition of these being conjugate
states, we must have the following two conditions for minimizing the potential
energy,
\begin{equation}\label{12}
\begin{array}{rcl}
0 &=& I + \alpha Q_1 + \frac{1}{2}\beta Q_1^2 + s Q_1^3\\[2mm]
0 &=& I + \alpha Q_2 + \frac{1}{2}\beta Q_2^2 + s Q_2^3\,,
\end{array}
\end{equation}
and a third condition that they possess the same potential energy value, 
\begin{equation}\label{Ex-eqn}
I Q_1 + \frac{1}{2}\alpha Q_1^2 + \frac{1}{6}\beta Q_1^3 + \frac{s}{4} Q_1^4 = \mathcal{H} =
I Q_2 + \frac{1}{2}\alpha Q_2^2 + \frac{1}{6}\beta Q_2^3 + \frac{s}{4} Q_2^4\,.
\end{equation}
The aim now is to solve the three equations (\ref{12}) and (\ref{Ex-eqn}) for
the three unknown $Q_1,Q_2,I$. Surprisingly, there is an explicit exact
solution. The details are lengthy and will be skipped.  The result is
\begin{equation}\label{x-y-finalform}
Q_{1} = - \frac{\beta}{6s} \pm \frac{1}{s} \sqrt{D}\qand Q_2 = - Q_1 -\frac{\beta}{3s}\,,
\quad D = \frac{1}{12}\beta^2 - s\alpha\,,
\end{equation}
where the sign in the first term is chosen so that $Q_1<Q_2$. Thus, we find a
necessary condition for the heteroclinic connections to exist by imposing our
roots must be real, requiring that
\begin{equation}\label{root-reality-cond}
\frac{1}{12}\beta^2-s\alpha \geq 0 \quad \Rightarrow \quad  A_{kk}(k_0)^2 \geq 2A_k(k_0)A_{kkk}(k_0)\,.
\end{equation}
Back substitute a solution to find expressions for the action $I$ and
$\mathcal{H}$,
\begin{equation}\label{I-alpha-beta}
I =   \frac{\beta}{108 s^2}\left( 18 s \alpha -  \beta^2\right)\,,\quad
\mathcal{H}=  -\frac{(18s\alpha - \beta^2)^2}{1296s^3} = -\frac{9s}{\beta^2}I^2\,.
\end{equation}
Hence for any $(\alpha,\beta)\in\R^2$ with $\alpha,\,\beta  = o(\eps^{-1})$,
these formulas give the values of $I$ and $\mathcal{H}$ at which conjugate
periodic orbits $Q_1$ and $Q_2$ in (\ref{x-y-finalform}) exist.

This gives the existence of conjugate periodic orbits. The next necessary condition
for a heteroclinic connection is that both conjugate periodic orbits be
hyperbolic. Hyperbolicity is determined by linearizing (\ref{ODE-form}) about
either constant state, denoted by $Q_0$,
\begin{equation}\label{ODE-form-lse}
\mathscr{K} \delta Q_{XX} + \big(\alpha + \beta Q_0 + 3sQ_0^2\big)\delta Q =0\,.
\end{equation}
Hence a basic state is hyperbolic when
\[
{\rm sign}\Big( \mathscr{K} \mathcal{H}''(Q_0)\Big)<0\,.
\]
Substitution of either $Q_{1}$ or $Q_2$ into this condition reveals the two
solutions are hyperbolic when 
\[
\frac{\mathscr{K}}{s}\left(\frac{1}{12}\beta^2-s\alpha\right)<0\,.
\]
As the bracketed term is assumed positive for the existence of the conjugate
states, we therefore have the condition for hyperbolicity as
\[
{\rm sign}(\mathscr{K}s) \equiv {\rm sign}\big(\mathscr{K}A_{kkk}(k_0)\big)<0\,,
\]
i.e. that the underlying stationary modified KdV is defocussing, in line with
previous studies on the stability of solutions within equations of this type
\cite{Bronski2016,Kamchatnov2012,Driscoll1977}. This observation, as well as the
condition for existence, are summarized as follows.

Let $\hat{u}(k_0x+\theta_0) \equiv
    \hat{u}(z)$ be a periodic solution of (\ref{primary-ode}). Suppose a local
    wavenumber perturbation $k$ is introduced such that $|k-k_0|\ll 1$. Then,
    provided that
    \[
    A_k(k_0) = \mathcal{O}(|k-k_0|^2)\,, \quad  A_{kk}(k_0) = \mathcal{O}(|k-k_0|)\,,
    \]
    a hyperbolic pair of conjugate states representing a heteroclinic connection
    between two periodic states emerges from the point $A_k(k_0) = A_{kk}(k_0) =
    0$ provided that
    \[
    A_{kk}(k_0)^2 \geq 2A_k(k_0)A_{kkk}(k_0) \quad {\rm and} \quad \quad {\rm sign}\big(\mu''''(0;k_0)A_{kkk}(k_0)\big)<0\,,
    \]
    where $\mu(\nu;k_0)$ is the Bloch spectrum of the wavetrain $\hat{u}(z)$.

To finish our discussion here, we provide the explicit form of the heteroclinic
connection between the two wavetrains as prescribed by the normal form
(\ref{ODE-form}). This can most readily be obtained by multiplying this equation
by $q_X$ and integrating once with respect to $X$ and prescribing the conditions
for a repeated double root of the resulting quartic polynomial (requiring
(\ref{root-reality-cond}) to hold) and assuming $\mathscr{K} s <0$. In such a
case, we end up with the first order ODE
\[
q_X = \pm \sqrt{-\, \frac{s}{4\mathscr{K}}} \ (q-Q_1)(Q_2-q)\,, \quad Q_1<q<Q_2\,.
\]
This has a pair of tanh solutions, one for each sign of the right-hand side:
\[
q(X) = \frac{(Q_1+Q_2)}{2} \pm \frac{ (Q_2-Q_1)}{2}{\rm tanh}\big(\delta X\big)\,,\quad \delta = \sqrt{\left\vert\frac{s}{4\mathscr{K}} \right\vert} \ {}\frac{(Q_2-Q_1)}{2\sqrt{2}}\,,
\]
This gives a prototypical heteroclinic connection between two asymptotic
wavenumbers $k_0+\eps Q_1$ and $k_0+\eps Q_2$. The strength of the connection
(i.e. its amplitude) is precisely the difference between the two roots. The
above form also highlights that the strength of the connection, the strength of
the cubic term in (\ref{ODE-form}) and the Bloch spectrum all control the width
of the heteroclinic connection. The resulting perturbed wave phase is
\[
\phi(X) = \int Q(Y)dY = \frac{(Q_1+Q_2)}{2} X\pm\frac{Q_2-Q_1}{2\delta}\ln \big(\cosh(\delta X)\big)+\phi_0 \equiv \phi_0+\eps \frac{Q_1+Q_2}{2} x+\frac{Q_2-Q_1}{2\delta}\ln \big(\cosh(\eps \delta x)\big)\,.
\]
The phase is unbounded as $x\to\pm\infty$, but it is asymptotically consistent and of lower order than the original wave phase:
    \[
    \lim_{\eps \to 0} \left\vert \frac{\phi}{z}\right \vert <\lim_{\eps \to 0} \frac{\eps Q_2}{2k_0} \to 0
    \]
Finally, we note that this approximation and explicit solution do not give us the selected asymptotic phases of the periodic orbits, and geometrically corresponds to the
special case where the two cylinders shown in Figure~\ref{fig-manifolds} are
identical and intersect perfectly such that the two circles on the Poincar\'e
section overlap, due to integrability of the normal form. In order to capture the splitting of the circles on the
Poincar\'e section asymptotically, we would need to go to higher order and
possibly beyond all orders and we leave this for future work. 

\section{The SH357 equation}
\setcounter{equation}{0}
\label{sec-SH357-review}

The next two sections are based on the SH357 equation introduced by \textsc{Knobloch et al.}~{\cite{kuw19}}. It is of the form (\ref{primary-ode}) but with a higher order nonlinearity
\begin{align}\label{e:SH357}
   u_{xxxx} + 2u_{xx} + (1-\mu)u + au^3 - bu^5 + u^7=0\,,
   \end{align}
replacing $\lambda$ in \cite{kuw19} by $\mu$ and noting that $\sigma=2$.
Our main new result is the discovery of two codimension 2 points in this system, one of which has the right signs and produces an organizing center for a branch of heteroclinic connections.  Before presenting the
results, we review some of the relevant results in \cite{kuw19}. Their strategy for computing families of periodic solutions, as a precursor to connecting them, is very different from this paper in an interesting way.  They work from a bifurcation perspective, with bifurcation parameter $\mu$, whereas we work from a symplectic perspective, with emphasis on the action, Hamiltonian, and wavenumber spaces.

We have re-calculated some of their results for validation. Figure \ref{fig-kuw} shows our recalculation of Figure 1, Panel (a) in \cite{kuw19}.  It shows a branch of periodic orbits plotted in the $\mu$ versus $\|u\|$ plane on the left, and the Hamiltonian versus $\mu$ on the right. Our results match \cite{kuw19} to graphical accuracy.  An interesting property of the left diagram is that the turning points correspond to a change in temporal stability and give no information about spatial Floquet multipliers.  This latter observation follows by noting that periodic solutions are critical points of
\begin{equation}\label{E-def}
\mathcal{E}(u,\mu) =\int_\Omega \left[ \fr u_{xx}^2 +uu_{xx} - \fr(\mu-1)u^2 - F(u)\right]\,\rd x
\end{equation}
with $F'(u)=-au^3 + bu^5 - u^7$, and $\Omega\subset\R$ (and $\Omega=S^1$ in the periodic case).  A key functional is
\[
-\mathcal{E}_\mu = \fr \|u\|^2\,.
\]
\textsc{Maddocks}~\cite{m87} proves, for general functionals of the form (\ref{E-def}), that turning points in the $(\lambda,-\mathcal{E}_\mu)$ plane correspond to changes of (temporal) stability.  This agrees with Figure \ref{fig-kuw} (and Figure 1(a) in \cite{kuw19}), since $\|u\|=\sqrt{-2\mathcal{E}_\mu}$. Note also that the plots of $H$ versus $\mu$ are dramatically different from $H$ or $A$ versus $k$. In the $(H,\mu)-$plane, turning points appear as cusps. This feature is predicted by gradient $\mu-$based bifurcation theory (compare with Figure 3 of \textsc{Thompson}~\cite{thomp79}).
\begin{figure}[ht]
\begin{center}
\includegraphics[scale = 1]{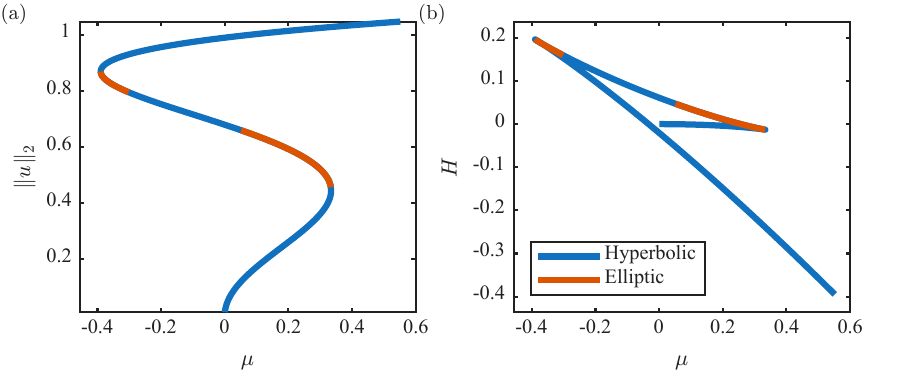}
\caption{Our re-calculation of Figure 1, Panel (a), from \textsc{Knobloch et al.}~\cite{kuw19}.  Turning points on the left panel correspond to changes in temporal stability.  We have added in our calculation of spatial Floquet multipliers, with blue indicating hyperbolic Floquet multipliers.}
\label{fig-kuw}
\end{center}
\end{figure}
We cannot just re-plot the $(H,\mu)$ plots in the $(H,k)$ plane because the wavenumber is fixed in the $(H,\mu)$ plane.  We will approach the problem afresh using our strategy of plotting $H$ (or $A$) versus wavenumber as this gives {\it bifurcation of spatial Floquet multipliers at turning points}. This strategy will be all the more important when we search for codimension 2 points as they appear naturally in the $(A,k)$ and $(H,k)$ space and do not appear in $\mu-$based planes.

With three parameters and strong nonlinearity, the SH357 equation provides an excellent test case for finding codimension two points.

\subsection{Finding and continuing codimension two points}
\label{sec-numerics-codimtwo}

In this subsection, we present a numerical algorithm for computing the codimension
two points, $A_k=A_{kk}=0$, in a field of periodic orbits. The algorithm is more general (and is applied to the ac-Boussinesq in \S\ref{sec-boussinesq}), but will be developed in the context of the SH357 equation. Step one is to identify a codimension 1 singularity; that is, a line or surface $A_k(k,\cdot)=0$, where the dot represents other parameters.  Step two is to continue this singularity in a
two-parameter space until a codimension 2 point is hit. The third step is to calculate heteroclinic connections in the unfolding of the singularity.

We will use
$\hat{u}(z,k,\mu,a,b)$ to denote a multiparameter field of $2\pi$-periodic orbits (sometimes suppressing the parameters for brevity). 
In step one, the parameters $\mu$ and $a$ are fixed, and then we work in a two-parameter family of periodic orbits parameterized by $k$ and $b$. We compute a codimension 1
curve in the $(k,b)$ space, parameterized by $s$, in some interval
$s_1<s<s_2$. The parameterized curve $(k(s),b(s))$ is defined by
\begin{equation}\label{eq:ak-curve}
\partial_k A(k(s),b(s))=0\quad\mbox{(with $a,\mu$ fixed)}\,.
\end{equation}
Figure \ref{fig:codim1-SH357} shows various perspectives on the set defined by (\ref{eq:ak-curve}) in the case $\mu = 0.3$ and $a = 1.5$.
\begin{figure}[ht]
\begin{center}
\includegraphics[scale = 1]{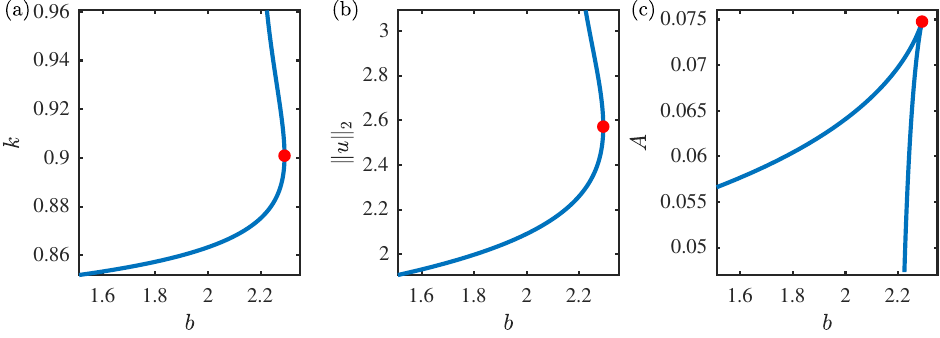}
\end{center}
\caption{(a) Plot of the curve $(k(s),b(s))$ generated by solving
  (\ref{eq:ak-curve}) for the family of periodic orbits of (\ref{e:SH357})
  with $\mu =0.3$ and $a = 1.5$. (b) A plot of the $L^2$ norm of the solution
  along the same curve as (a). (c) A plot of the action, $A$, along the same
  solution curve}
\label{fig:codim1-SH357}
\end{figure}

In principle, the codimension two point can be identified by computing $\partial_{k}^2A(k,b)$ along the codimension one curves looking for points where it changes sign.  However, we will show that there is an easier way. The relation of
the geometry of the $(k,b)-$curve to the action derivatives is summarized in
the following. 

\begin{lem}\label{lemma5.1} {\it Suppose that the curve $(k(s),b(s))$ satisfying
  $A_k(k(s),b(s))=0$ is smooth (at least differentiable), for $s_1<s<s_2$. For
  some $s^*\in(s_1,s_2)$, suppose that $b'(s^*)=0$, $k'(s^*)\neq0$, and
  $A_{kb}(k(s^*),b(s^*))\neq0$, then $(k(s^*),b(s^*))$ corresponds to a
  codimension two point:}
\begin{equation}\label{cod-two-point}
  A_k(k(s),(s)) = A_{kk}(k(s),b(s))=0\quad \Leftrightarrow\quad
  k'(s) \neq0\ \mbox{and}\ b'(s)=0\,,\quad\mbox{for some}\ s^*\in(s_1,s_2)\,.
\end{equation}
\end{lem}

\noindent{Proof.} Since the curve is smooth, the function
\[
f(s) := A_k(k(s),b(s)) \,,
\]
is zero for each $s$ and differentiable.  Therefore
\[
0 = f'(s) = A_{kk}(k(s),b(s))k'(s) + A_{kb}(k(s),b(s))b'(s)\,.
\]
The claim (\ref{cod-two-point}) is then immediate.$\hfill\square$
\vspace{.15cm}

In Fig. \ref{fig:codim1-SH357}, the codimension two point found with this Lemma is identified with a red dot. It is associated with a fold point of the curve in the $(k,b)$-plane. 

\begin{rmk} The condition $A_{kb}\neq0$ is generic.  To see this, set
up a moving orthonormal frame on the $(k,b)$ curve.  The tangent and a normal
vectors are
\[
{\bf t}(s) = \frac{1}{\ell(s)}\begin{pmatrix} k'(s)\\ b'(s)\end{pmatrix}
\qand {\bf n}(s) = \frac{1}{\ell(s)}\begin{pmatrix} b'(s)\\ -k'(s)\end{pmatrix}\,,
\]
with $\ell(s) = \sqrt{k'(s)k'(s)+b'(s)b'(s)}$.  The gradient of $A_k(k,b)$, now
considered as a field on $(k,b)$ space, in the direction of ${\bf n}$ at
$s=s^*$, is
\[
  {\bf n}\cdot\nabla A_k = \frac{1}{\ell(s)}\begin{pmatrix} b'(s)\\ -k'(s)\end{pmatrix}\cdot\begin{pmatrix} A_{kk}\\ A_{kb}\end{pmatrix}  =
  -\frac{1}{\ell} k'(s) A_{kb}\Big|_{s=s^*}\,.
  \]
  Hence, with $k'(s^*)\neq0$, $A_{kb}$ is proportional to the normal derivative
  of $A_k$.  Now, the curve (\ref{eq:ak-curve}) is the zero level set of $A_k(k,b)$
  and, in the direction normal to that curve, the level sets on each side are
  generically non-zero level sets. It is in this sense that
  $A_{kb}(k(s^*),b(s^*))\neq0$.
  \end{rmk}

The above approach appears to be the simplest strategy for finding the $A_k=A_{kk}=0$ codimension two points.  However, such codimension two points can be identified in other ways
  using different parameter spaces. Indeed, the most straightforward way is to
  compute $A_{kk}$ along the curve (\ref{cod-two-point}), using the formula for
  $A_{kk}$ in equation (\ref{Akk-calc}) in Appendix \ref{app-a}, and monitor
  when it passes through zero.  Interestingly, this strategy, when followed in
  the $(A,b)-$plane, produces a cusp at the codimension two point, so the point
  is easily identified. 
  
  The algorithm that implements Lemma \ref{lemma5.1} is based on solving the system 
    \begin{subequations}
    \begin{align}
    	L(k) \hat{u} + f(\hat{u}) + cu_z =& 0,\\
	L(k)\hat{u}_k + L'(k)\hat{u} + f'(\hat{u})\hat{u}_k + c_k\hat{u}_{kz}=& 0,\\
	A_k(\hat{u},\hat{u}_k) =& 0
    \end{align}
    \end{subequations}
    for $(\hat{u},\hat{u}_k,c,c_k)$ and an additional equation parameter (e.g. $b$) on $z\in[0,2\pi)$ along with periodic boundary
    conditions and phase conditions
    \[
    \int_{0}^{2\pi}u^{\mbox{old}}_z(u - u^{\mbox{old}}) \mathrm{d}z = \int_{0}^{2\pi}u^{\mbox{old}}_z(u_k - u^{\mbox{old}}) \mathrm{d}z = \int_{0}^{2\pi}u^{\mbox{old}}_z(u_{kk} - u^{\mbox{old}}) \mathrm{d}z = 0.
    \]
This allows us to continue curves of periodic orbits where $A_k=0$ and one can
then look for a cusp as another equation parameter is varied. Once the cusp
point is located, we then add two more equations,     
   \begin{align*}
	L(k)u_{kk} + 2L'(k)u_k + L''(k)u + f'(u)u_{kk} + f''(u)(u_k)^2 + c_{kk}(u_{kk})_z =& 0,\\
	L(k)u_{kkk} + 3L'(k)u_{kk} + L'''(k)u + 3L''(u)u_k + f'(u)u_{kkk}+f'''(u)u_k^3 + 3f''(u)u_ku_{kk} + c_{kkk}(u_{kkk})_z =& 0,\\
	A_{kk}(u,u_k,u_{kk}) =& 0,
	\end{align*}
 for $(u_{kk},u_{kkk},c_{kk},c_{kkk})$ with a further two phase conditions
 similar to the above, and another equation parameter (in this case, either $\mu$ or
 $a$), in order to trace out two parameter diagrams where $A_k=A_{kk}=0$. Figure
 \ref{fig:codim2-SH357} shows the result of this computation, where $a$ is
 treated as a continuation parameter for the fixed value $\mu = 0.3$. These
 codimension two curves emanate from the red dot shown in the pane of Fig.
 \ref{fig:codim1-SH357}, i.e. the values of $b$ and $k$ are given when $a =
 1.5$. In Figure \ref{fig:codim2-SH357}(b), we indicate two points on the branch of codimension 2 curves, $p_{1,2}$. The action, $A(k)$, is computed near these codimension 2 points, and this computation is shown in panel \ref{fig:codim2-SH357}(c). In this panel, the dashed curves are the plot of the action near $p_1$ and the solid curve is the plot of the action near $p_2$. 

\begin{figure}[H]
    \centering
    \includegraphics[scale = 1]{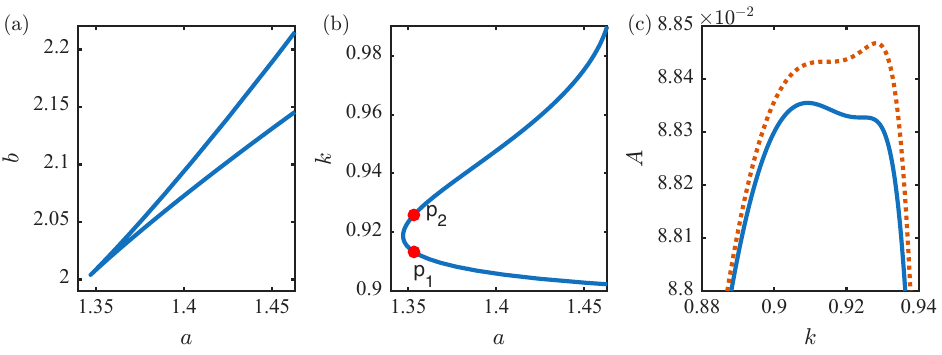}
    \caption{ (a),(b) Parameter space for the SH357 equation \eqref{e:SH357} with $\mu=0.3$ at which
    codimension 2 points occur. (c) Plot of the action, $A(k)$, near the points $p_1$ and $p_2$ labeled in panel (b), the dashed curve is near the point $p_1$ and the solid curve is contains $p_2$.}
    \label{fig:codim2-SH357}
\end{figure}

We now further probe the parameter space near the codimension two points. We summarize
these results in Figure \ref{fig:Ham_near_codim2}, where we plot an example of the
Hamiltonian \eqref{H-def}  as a function of $k$. Figure
\ref{fig:Ham_near_codim2}(a) shows the structure of the Hamiltonian as the
coefficient of the cubic term, $a$, is varied. The value of $b = 2.15$ is
fixed, and is chosen to coincide with the values obtained from the parameter
continuation in Figure \ref{fig:codim2-SH357}. The curve with parameter values near the codimension-two point
is shown in black, while the curves of other colors have
only slight changes in the parameter values. The relatively small parameter
range for which such points exist illustrates the careful computational approach
outlined in this section. Figure~\ref{fig:Ham_near_codim2}(b) is a zoom in of the  small boxed region from panel (a) to illustrate the coalescence and annihilation
of the critical points of the Hamiltonian as the parameters are varied. 
\begin{figure}[H]
    \centering
    \includegraphics[scale = 1]{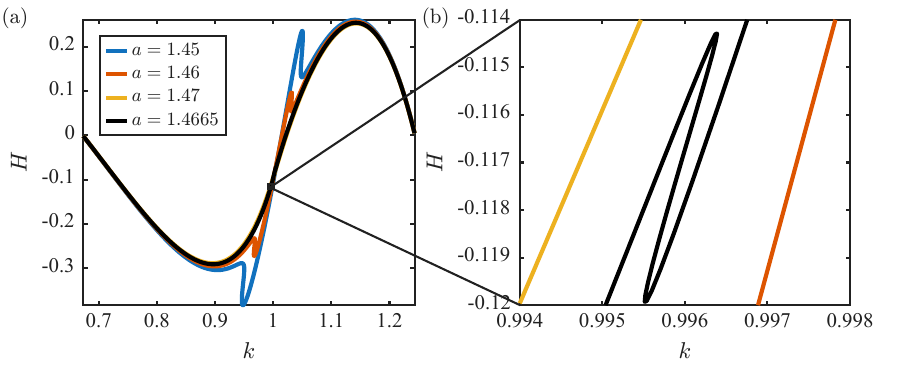}
    \caption{Behavior of the spatial Hamiltonian near the codimension 2 point for the SH357 equation with $\mu =0.3$ and $a = 1.45$. }
    \label{fig:Ham_near_codim2}
\end{figure}

\subsection{Heteroclinic connections near codimension two points}
\label{sec-SHE_PtoP}

In this section, we will outline how we compute the connecting fronts emanating
from the codimension two point found in \S\ref{sec-numerics-codimtwo}. Given
the discussion on the dimension of the intersection of the unstable and stable
manifolds of the periodic orbits in \S\ref{subsec-ubar-connect}, there are some choices in the set up the numerical methods in terms of which asymptotic parameters we wish to
select. In particular, we need to find two
parameters from the set $(k_-,k_+,\phi_-,\phi_+)$ if we do not use the
Hamiltonian restriction. In what follows, we briefly outline a numerical implementation of the
far-field core decomposition for the connecting fronts. The context will be general PDE systems whose connecting fronts select two parameters and then we
describe how one adapts this for the conservative case. 

We consider (\ref{primary-ode}) in the abstract form
\begin{equation}
L(\partial_x)u + N(u) = 0,
\end{equation}
where $L$ is a linear and $N$ a nonlinear operator.  The proposed far-field core decomposition of $u(x)$ is
\begin{equation}\label{e:far-core}
u(x) = u_-(k_-x + \phi_-)\chi_-(x) + w(x) + u_+(k_+x+\phi_+)\chi_+(x)\,,
\end{equation}
where $w(x)$ is the remainder function and $u_\pm(z)$ are periodic orbits satisfying
\begin{equation}\label{e:periodic-prob}
L(k_\pm\partial_z)u_\pm + N(u_\pm) + c_\pm (u_\pm)_z= 0,\qquad u_{\pm}(z+2\pi) = u_\pm(z),\qquad \int_0^{2\pi}u^{\mbox{old}}_z(u_\pm-u^{\mbox{old}})dz = 0\,.
\end{equation}
Upon substituting (\ref{e:far-core}) into equation (\ref{primary-ode}) we find
an equation for the core $w(x)$ given by
\begin{align}\label{e:core-eqn}
Lw + \left(w+\sum_{\pm}u_\pm\chi_\pm \right) - \sum_{\pm}\left(Lu_{\pm} + N(u_\pm) \right)\chi_\pm = 0\,.
\end{align}
We solve this equation on a large finite domain $x\in[-L,L]$, and add two phase conditions 
\begin{equation}\label{e:phase}
\int^{L}_{\pm x = L - 2\pi/k_\pm}w(x) \cdot \partial_x u_{\pm}(x;k_{\pm},\phi_\pm)dx = 0\,.
\end{equation}
In summary, we solve (\ref{e:periodic-prob}), (\ref{e:core-eqn}) and (\ref{e:phase}) for fixed $(k_-,\phi_-)$ and $(u_{\pm},w,k_+,\phi_+)$ if we wish to fix the value of the Hamiltonian level set, or we fix the phases $(\phi_-,\phi+)$ and solve for $(u_{\pm},w,k_\pm)$. 

\begin{rmk} In order to restrict to a Hamiltonian level set, the
integral condition for $u_-$ is changed to the Hamiltonian constraint. This will
restrict the periodic orbits to a specified Hamiltonian level set (i.e, this
selects a $k_-$) and we then also fix $\phi_-$. Either way, the two periodic orbits and the core will lie on \emph{some} Hamiltonian level set.
\end{rmk}

Numerically, we discretize the
equations for $u_\pm$ via a Fourier pseudo-spectral method~\cite{Trefethen2000}
and use fourth-order finite differences for the $w$ equation. They are implemented in \textsc{Matlab}, and numerical continuation is implemented using \textsc{Avitabile et al.}~\cite{Avitabile2020}. 

Based on the computation shown in Figure \ref{fig:codim2-SH357}, we seek a
heteroclinic connection in the unfolding of the codimension two point by selecting the
parameters 
$$\mu = 0.3, \quad  a = 1.45 , \quad  b = 2.15. \ $$
A front is then computed between two periodic orbits by fixing the phases at the left and right ends of the domain (set to zero) and solve for the wavenumbers while varying $a$ to approach the Hamiltonian codimension 2 point at $a\approx1.466$. In Figure~
\ref{fig:wavenum_and_sol_near_c2}, we plot the result of this front computation in Panel (b).  In Panel (a) we show that the wave numbers converge to the same value as we get close to the co-dimension 2 point. 
\begin{figure}[H]
    \centering
    \includegraphics[scale =1 ]{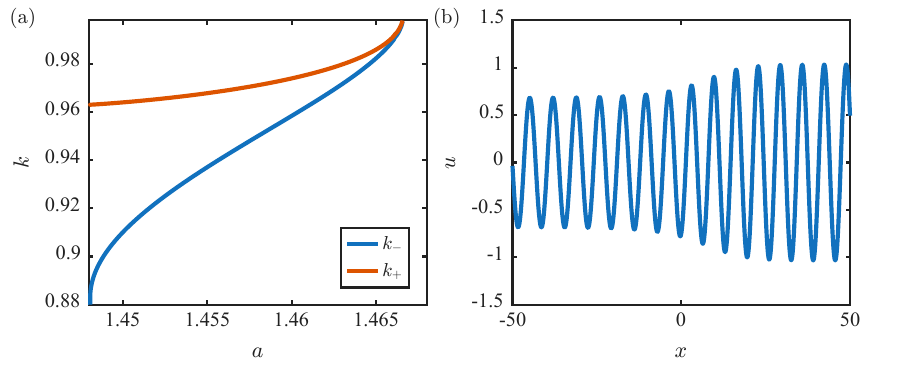}
    \caption{Wavenumber of the twp periodic orbits in (a) as a function of the parameter $a$, coalescing at the codimension two point ($a\approx 1.466$). The right panel (b) shows a connecting orbit at $a=1.45$ between two distinct periodic orbits in the unfolding of the codimension two point.}
    \label{fig:wavenum_and_sol_near_c2}
\end{figure}

In summary, we have established an algorithm for finding a codimension two point in a field of periodic solutions, and have shown how heteroclinic connections are found in the unfolding of the codimension two point, as predicted by the normal form theory.  In the next section we apply this strategy to a different PDE.

\section{Fronts near codimension-two points in a Boussinesq system}
\setcounter{equation}{0}
\label{sec-boussinesq}

The purpose of this section is three-fold: to show another example where a codimension two point can be found, to study a system that does not reduce to the SH form (\ref{primary-ode}), and to compute oscillatory fronts in a Boussinesq system for the first time. The strategy is to first reduce the Boussinesq system (\ref{ac-B}) to a steady system on $\R^4$, identify the Lagrangian, Hamiltonian, and action functionals, find a family of periodic orbits, and then implement our algorithm for zooming in on a codimension two point.

Relative to a moving frame, $x\mapsto x-Ct$, (\ref{ac-B}) is
\begin{equation}\label{ac-B-sec7}
-Ch_x + (u + hf'(u))_x +a u_{xxx} =0\qand -C u_x + f(u)_x + gh_x + c h_{xxx} = 0\,.
\end{equation}
Integrating once, and using the function $F(h,u)$ defined in (\ref{ac-F-def}), it reduces to the conservative system on $\R^4$,
\begin{equation}\label{ac-B-steady-sec7}
c h_{xx} = F_h\qand a u_{xx} = F_u\,.
\end{equation}
The Lagrangian and Hamiltonian densities for this system are
\begin{equation}\label{LH-def-bouss}
L = \fr a u_x^2 + \fr c h_x^2 + F(h,u)\qand 
H = \fr a u_x^2 + \fr c h_x^2 - F(h,u)\,.
\end{equation}
The equations (\ref{ac-B-steady-sec7}) can be expressed as a canonical Hamiltonian system, but the structure won't be needed here, just the key invariants.  The action density, for an ensemble of solutions, $h(x,s),u(x,s)$ parameterized by $s\in\R$, is
\begin{equation}\label{action-bouss}
A = a u_x u_s + c h_x h_s\,.
\end{equation}
It satisfies the action conservation law 
\begin{equation}\label{action-claw}
\frac{\partial A}{\partial x} = \frac{\partial L}{\partial s}\,.
\end{equation}
When $h,u$ are periodic in $s$, the integral of $A$ over $s$ is a relative integral invariant.

To compute codimension 2 points
in the $ac-$Boussinesq system \eqref{ac-B}, we follow a similar
approach to that outlined in Section \ref{sec-numerics-codimtwo} for the
Swift-Hohenberg equation \eqref{primary-ode}. For these computations, we fix the constants of integration to be $\iA = -2$ and $\iB = 2$, and set $C=0$. The periodic orbits of wavenumber $k$ are denoted by
\begin{equation}
  u = \uh(z,k) \quad h = \hh(z,k) \quad z = kx,
\end{equation}
and they satisfy the equations
\begin{equation}\label{eq:steady-ac-bouss}
a k^2 \uh_{zz}  + \uh + \hh \left(\uh + \alpha \uh^2\right) = -2 \qand c k^2 \hh_{zz} + \hh  + \frac{1}{2}\uh^2 + \frac{\alpha}{3} \uh^3 = +2\,.
\end{equation}
The action of this family of periodic orbits is
\begin{equation}\label{action-B-per}
A(k) = \frac{1}{2\pi} \int_0^{2\pi} \big(ak\uh_z^2 +ck\hh_z^2\big)\,\rd z\,.
\end{equation}
The codimension two point occurs when ${A}_k=A_{kk}=0$ (equivalently $H_k=H_{kk}=0$).

Periodic solutions are computed on the domain $z \in [0,2\pi)$, and gradients with respect to $k$ are computed using an approach similar to that in Section \ref{sec-numerics-codimtwo}. Codimension 2 points are found by solving for values of $k$ and $\alpha$ such that $A_k = A_{kk} = 0$ for fixed values of the dispersion coefficients $a$ and $c$. For the sake of this computation, we will fix values of $c < 0$ and treat $a$ as a continuation parameter. 

In Figure~\ref{fig:bouss_codim2}, we show the continuation of the co-dimension 2 point for various values of $c$ as $a$ is varied. In panel (a), we show the wavenumber at the codimension 2 point, and in panel (b) the corresponding selected $\alpha$ value. We plot the spatial Hamiltonian near the co-dimension 2 point for $a=0.2$ and $c=-0.3$ in the third panel of Figure~\ref{fig:bouss_codim2} The geometry of the $H-k$ plot is similar  to that in the SH357 equation in the previous section. 
\begin{figure}[p]
\begin{center}
\includegraphics[scale = 1]{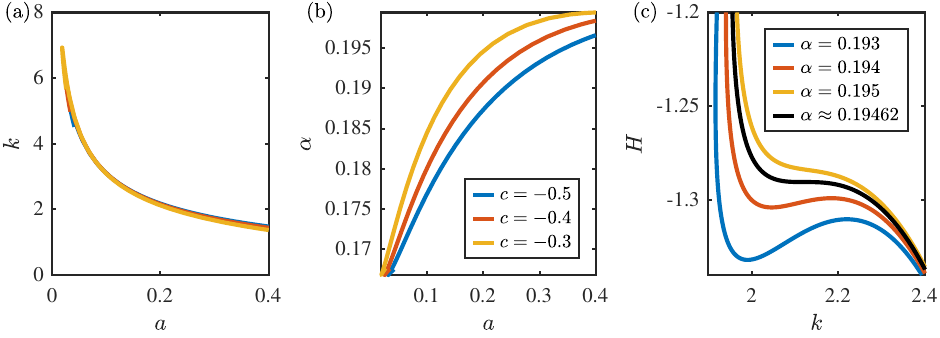}
\caption{(a), (b) Parameter values for codimension two points for the steady ac-Boussinesq equation \eqref{eq:steady-ac-bouss}. (c) Spatial Hamiltonian near the codimension two point with $a = 0.2$ and $c = -0.3$ for various $\alpha$ showing the emergence of the codimension-two point.}
\label{fig:bouss_codim2}
\end{center}
\end{figure}

To compute the fronts, that emerge from the codimension two bifurcation point, we fix the phase of the far-field periodic orbits to be zero, and solve for the wavenumbers of the far-field periodic orbits. This strategy was also used in some of the SH357 calculations.  The fronts in this case are computed using the farfield-core decomposition.

To find the codimension-two point, the strength of the highest order nonlinearity, $\alpha$, is incremented until we reach a codimension 2 point at $\alpha \approx 0.1946$.  At the codimension 2 point, the two wavenumbers coalesce to $k = k_- = k_+ \approx 2.12$. The wavenumbers of the far-field periodic orbits are shown in Fig. \ref{fig:codim2_bif_ac}(a) as a function of $\alpha$. We show a representative front solution near the codimension-two point with $\alpha = 0.194$ in Fig. \ref{fig:codim2_bif_ac}(b) and (c). For this computation, the phase of the periodic solutions are fixed to be zero. For illustrative purposes, we show a similar family of PtoP solutions in Fig. \ref{fig:codim2_bif_ac2}, where the phase of the periodic orbit near $+ \infty$ is fixed to be $-0.25$. 

\begin{figure}
    \centering
    \includegraphics[width=0.8\linewidth]{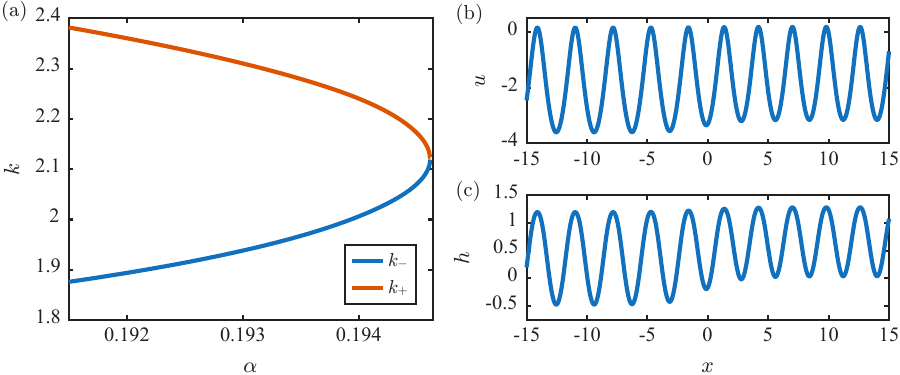}
    \caption{(a) Wavenumber of the far-field orbits that comprise the front solution of the $ac$-Boussinesq equation with $a = 0.2$ and $c = -0.3$ that bifurcate from the codimension two point. Panels (b) and (c) show two representative fronts near the codimension two point with $\alpha \approx 0.194$. }
    \label{fig:codim2_bif_ac}
\end{figure}

\begin{figure}
    \centering
    \includegraphics[width = 0.8\linewidth]{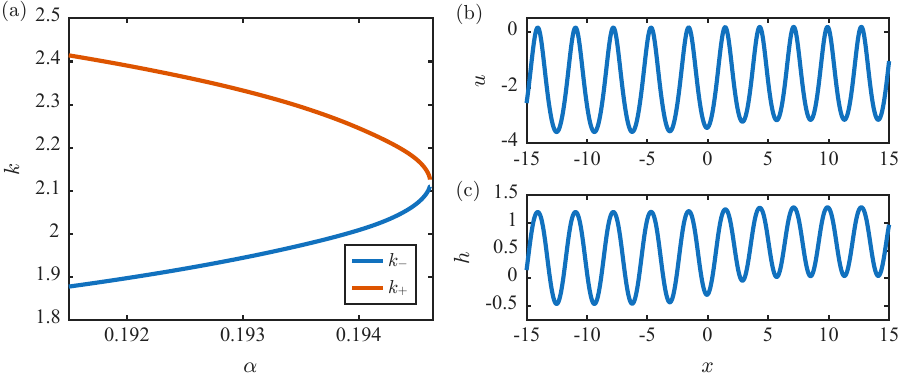}
    \caption{Same as Fig. \ref{fig:codim2_bif_ac}, where the phase of the periodic orbit at $+ \infty$ is fixed to $-0.25$. }
    \label{fig:codim2_bif_ac2}
\end{figure}

\begin{rmk} For the case of water waves ($\alpha=0$ in the nonlinearity), the strategy would be to continue the family of fronts from $\alpha\approx 0.194$ down to $\alpha=0$.  On the other hand, there are models in the theory of water waves where the higher-order nonlinearity occurs, an example being the Boussinesq equations that have been derived at the interface between two fluids of different densities (e.g. Nguyen \& Dias \cite{nd08}).
\end{rmk}

\section{Concluding remarks}\label{sec-con}

In this paper, we have clarified the geometric structure of PtoP heteroclinic
connections in conservative systems in 4 phase space dimensions. This viewpoint
makes it clear how the asymptotic phases are selected and why it is crucial to
compute them in order to find heteroclinic connections and establish multiplicity. Moreover, it points to the geometry of the (un)stable manifolds as tubes, at least in the hyperbolic case (in contrast to the inverse hyperbolic). These observations build on extensions of previous work. We have taken the subject in a new direction by introducing the role of action, in both the numerics and in the normal form theory, and demonstrated the theory with examples. The above results point to a number of new directions. 

\noindent{\bf Dimension counting.} In phase space dimension greater than 4, dimension counting can be used to ascertain whether there will be selection of the asymptotic phases. For instance, in the 4D case, we can see the left periodic orbit has a 3-dimensional unstable manifold (coming from the two unit Floquet multipliers
and one unstable multiplier) with corresponding co-dimension $= 4-3 = 1$. The
right periodic orbit has a $3-$dimensional stable manifold (coming from the two
unit Floquet multipliers and one stable multiplier) with corresponding
co-dimension $= 4-3 = 1$. The sum of the co-dimensions is $2$. Hence, the
dimension of the intersection is given by the dimension of the phase space minus
the sum of the co-dimensions, i.e. $4-2 = 2$ and so we need two parameters to be selected for the intersection to occur. Since the periodic orbits have to lie in the same Hamiltonian level set, for a given left periodic orbit with a fixed wavelength and phase, there will be a selected wavelength of the right periodic orbit due to the Hamiltonian, leaving the phase of the right periodic orbit to be selected as part of the intersection.  This dimension counting argument carries over to phase space dimension 6 and higher.  Indeed, some intriguing results on heteroclinic connections between both periodic orbits and quasiperiodic (QP) tori, in a 6-dimensional phase space, have been reported in \textsc{Baresi et al.}~\cite{bos18} and \textsc{Owen \& Baresi}~\cite{ob24}, with the latter showing that knot theory can be used to determine multiplicity of QP connections. 

\noindent{\bf Action and normal form.} A key contribution of the paper is the unfolding of the coalescence of two hyperbolic periodic orbits using modulation theory, resulting in the generic asymptotic equation \eqref{q-RE-unfolding-2}. This equation introduces a new and important quantity, the Action, for the understanding of the heteroclinic connections in Hamiltonian systems. The advantage of the modulation equation is that it helps one find where such heteroclinic connections will occur just by understanding the geometric properties of the periodic orbit at the coalescence point. While this asymptotic equation is universal, it unfortunately does not provide any information on the phase selection process at the codimension two point or the unfolding. We suspect this would require going beyond-all orders asymptotics. We have also described a systematic numerical method for locating the coalescence point and shown that this coalescence point occurs in two different systems. Hence, we believe this organizing center to be common. 

\noindent{\bf Space-time normal form.} Going back to the normal form (\ref{q-RE}), and including a slow time
$T=\eps^4t$, while also adding $u_t$ to (\ref{primary-ode}), results in a
time-dependent modulation equation, which is a Cahn-Hilliard equation,
\begin{equation}\label{ch-eqn}
q_T = \left( \mathscr{K}q_{XX} + \frrr A_{kkk} q^3\right)_{XX}\,.
\end{equation}
The Cahn-Hilliard equation has been derived before using phase modulation by
D\"ull~\cite{dull07} for a class of reaction-diffusion equations. In that case,
the nonlinearity in the normal form was quadratic. The advantage of the cubic nonlinearity is that it generates fronts in the normal form. A complete analysis of this time-dependent PDE is outside the scope of this
paper.  However, we do point to two directions of interest.  This equation could
be used to get a first approximation to the stability of the fronts found in the
normal form in \S \ref{subsec-nf-solutions}.  From results in the literature on
stability of fronts of the Cahn-Hilliard equation (e.g. \textsc{Bricmont et al.}~\cite{bkt99}), there is
reason to be optimistic that the fronts found in \S\ref{subsec-nf-solutions}
will be stable. However this stability would be for the core only, and does not take into account the stability of the states at infinity. There is also reason to be optimistic that the time-dependent equation
(\ref{ch-eqn}) will be valid, in the sense that solutions of it will stay
close to solutions of the time-dependent SH equation.  D\"ull \cite{dull07} proves the validity
of the quadratic Cahn-Hilliard equation, when compared with solutions of a
coupled reaction-diffusion equation.

\noindent{\bf Inverse hyperbolic periodic orbits.} As shown in Figure \ref{fig-hyperbolic}, there two cases of hyperbolic Floquet multipliers. In this paper we have focused on the positive case.  The negative case (inverse hyperbolic) appeared in Figure \ref{fig:bandara_comparison}, but only in the third disconnected periodic orbit. The inverse hyperbolic case needs to be treated differently. For example, the unstable manifold is not orientable and has the structure of a M\"obius band (e.g.\ Section 5 of \textsc{Meiske \& Schneider}~\cite{ms87} and Section 3 of \textsc{Aougab et al.}~\cite{abcdssw19}).  It is possible to connect hyperbolic and inverse hyperbolic periodic orbits by a heteroclinic in phase space dimension four, but the codimension two singularity here (coalescence of two periodic orbits) and associated normal form does not extend to coalescence of a hyperbolic and inverse-hyperbolic orbit, as a higher phase space dimension would be required.

\noindent{\bf Symplectic geometry.} Periodic orbits in Hamiltonian systems have a Maslov index.  There are several definitions of Maslov index in the literature (e.g. \textsc{Long}~\cite{long-book}) and some definition will be operational here.  Hence, an open question is the significance of the value of the Maslov index of each periodic orbit, or the difference between the two.  Another open question is the significance of the value of the asymptotic phases or the jump in value when crossing the surface of section.  Does it have a symplectic or other characterization?

\noindent{\bf Multi-space dimension PDEs.} Looking ahead, the methods presented in this paper can, in principle, be extended to PDEs in two space dimensions; for example, to 2D hexagon fronts on the infinite periodic strip. The 2D hexagon front asymptotics, should carry through in the similar fashion to locate ``stressed" hexagon connections and might explain the numerical investigations found on the plane~\cite{Subramanian2021}. The normal form theory can also capture, in principle, the coalescence of two hexagonal states.  Most of the numerical methods for finding the coalescence point and the far-field core method carry over to 2 space dimensions as well; see for instance~\cite{Lloyd2021,Lloyd2017}.  From an asymptotic point of view, the more interesting case is the quasi-periodic connection set up (for both ODEs in higher space dimension, as noted above, and for PDEs~\cite{Subramanian2016}), where there would be multiple `Actions' occurring. Unfolding that setup is likely to lead to new insights. Coalescence of two QP invariant tori can also be captured in a normal form.  This direction would build on the important results of \textsc{Owen \& Baresi}~\cite{ob24}.  Finally, it would be interesting to extend this theory to spatio-temporal wave train defects in Hamiltonian systems; that is, a multisymplectic version of \cite{Sandstede2004,Robert2025}. 

\subsection*{Data availability statement}
The data that support the findings of this study are openly available at \cite{Bridge2026Github}.

\section*{Acknowledgments}

TB, DR, and PS would like to thank the Isaac Newton Institute for Mathematical Sciences for support and hospitality during that programme {\it Dispersive Hydrodynamics}, held in Autumn 2022, supported by EPSRC Grant Number EP/R014604/1, where this work was initiated. The authors are grateful to Michael Shearer for many helpful discussions on the project during that programme.  The work of DL was partially supported by EPSRC Grant UKRI070. TB and DL are grateful to Nicola Baresi (Surrey Space Centre)  for illuminating discussions about related work in astrodynamics. DR and PS are grateful to Surrey Mathematics for funding visits to work on the project. For the purpose of Open Access, the authors have applied a Creative Commons Attribution (CC BY) public copyright licence to any Author Accepted Manuscript version arising from this submission.

\newpage
\begin{center}
\hrule height.15 cm
\vspace{.2cm}
--- {\Large\bf Appendix} ---
\vspace{.2cm}
\hrule height.15cm
\end{center}
\vspace{.25cm}

\begin{appendix}

\renewcommand{\theequation}{A-\arabic{equation}}
\section{\texorpdfstring{$k-$}~Derivatives of the action}
\label{app-a}
\setcounter{equation}{0}

In this appendix, we record formulas for the $k-$derivatives of the action
functional evaluated on a family of periodic orbits. For reference, the formula
for $A(k)$ is
\begin{equation}\label{A-def}
A = \frac{1}{2\pi}\int_0^{2\pi} \left(
 \sigma k \uh_z^2 - 2k^3 \uh_{zz}^2\right)\,\rd z\,.
\end{equation}
The derivatives $A_k$, $A_{kk}$ and $A_{kkk}$ are obtained by direct
calculation, so we just record the results. The first derivative is
\begin{equation}\label{Ak-calc}
A_k  = \frac{1}{2\pi} \int_0^{2\pi}\left(
\sigma \uh_z^2   + 2\sigma k \uh_z\uh_{zk} -6k^2\uh_{zz}^2 - 4k^3\uh_{zz}\uh_{zzk}\right) dz
\,.
\end{equation}
Differentiating a second time, and amalgamating terms
\begin{equation}\label{Akk-calc}
\begin{array}{rcl}
A_{kk}  &=& \frac{1}{2\pi} \int_0^{2\pi}\left(
4\sigma \uh_z\uh_{zk} + 2\sigma k \uh_{zk}^2
+ 2\sigma k \uh_z\uh_{zkk} -12 k\uh_{zz}^2 \right. \\[2mm]
&&\quad \left. -24k^2\uh_{zz}\uh_{zzk}
- 4k^3\uh_{zzk}\uh_{zzk}- 4k^3\uh_{zz}\uh_{zzkk}
 \right) dz\,.
\end{array}
\end{equation}
Differentiate a third time and simplify to get
\begin{equation}\label{Akkk-formula}
\begin{array}{rcl}
A_{kkk}  &=&\displaystyle \frac{1}{2\pi} \int_0^{2\pi}\Big(
6\sigma \uh_{zk}^2 +6\sigma \uh_z\uh_{zkk} 
 + 6\sigma k \uh_{zk}\uh_{zkk}  + 2\sigma k \uh_z\uh_{zkkk}\\[4mm]
&&\qquad\qquad -12 \uh_{zz}^2
 -72 k\uh_{zz}\uh_{zzk}  -36k^2\uh_{zzk}^2  - 36 k^2\uh_{zz}\uh_{zzkk} \\[4mm]
&&\qquad\qquad\qquad - 12k^3\uh_{zzkk}\uh_{zzk}  - 4k^3\uh_{zz}\uh_{zzkkk}
 \Big) dz\,.
\end{array}
\end{equation}
All of the $k-$derivatives of $A$ are quadratic in $\uh$, and can be expressed using the inner product.  For example,
\[
A = \sigma k \lth \uh_z,\uh_z\rth - 2k^3\lth\uh_{zz},\uh_{zz}\rth \,.
\]
The inner product representation will be important in the normal form theory.

 \renewcommand{\theequation}{B-\arabic{equation}}

 \section{Derivation of the Nonlinear Normal Form}
 \label{mod-reduction}
 In this section, we will sketch the derivation of the normal form for the emergence of a
 heteroclinic connection near the co-dimension two point, $A_k=A_{kk}=0$, using
 phase modulation.

Phase modulation for ODEs can be carried out using normal form transformations
(e.g.\ Iooss \cite{iooss1988global}, Mielke \cite{mielke1996spatial}).  For
PDEs, normal form transformations are cumbersome and therefore an ansatz
approach is used (e.g.\ Doelman, et al.~\cite{doelman2009dynamics}).  For
conservative systems, phase modulation for PDEs is based on Whitham modulation
theory which is built on an average Lagrangian (e.g. Bridges
\cite{bridges2017symmetry}).  Here, we use the ansatz approach even for ODEs, as
it is easier and will easily extend to the PDE case when we add in time
and derive the Cahn-Hilliard equation (cf.\ \S\ref{sec-con}). We will bring in the role of action {\it
a posteriori}.

The derivation of the normal form, as mentioned in the main text of the paper,
is achieved by substitution of the ansatz (\ref{modulation-ansatz}) into
(\ref{primary-ode}), expanding it in a Taylor series in $\eps$, and solving
order by order.  At zeroth order we just recover the equation for the periodic
solution, shifted in phase
\begin{equation}\label{zero-order-eqn}
k^4\uh_{zzzz}(z+\phi,k) + \sigma k^2\uh_{zz}(z+\phi,k) + V'(\uh(z+\phi,k)))=0\,.
\end{equation}
At first order, we find
\[
{\bf L}(0)\uh_k + 4k^3\uh_{zzzz} + 2\sigma k\uh_{zz} =0\,,
\]
where ${\bf L}(0)$ is defined in (\ref{L-lambda-def}). This equation is
satisfied exactly since it is just the derivative of (\ref{zero-order-eqn}) with
respect to $k$. At second order, we find
\[
{\bf L}(0)w_2 + q_X\left(4k^3\uh_{zzzk} + 6k^2\uh_{zzz} + 2\sigma k\uh_{zk} + \sigma \uh_z\right) = 0 \,.
\]
The solvability condition for this equation is that the term proportional to
$q_X$ should be orthogonal to $\uh_z$,
\[
0 = \left\langle \uh_z,4k^3\uh_{zzzk} + 6k^2\uh_{zzz} + 2\sigma k\uh_{zk} + \sigma \uh_z\right\rangle\,.
\]
Remarkably this solvability condition is precisely $A_k=0$. This result is
confirmed by differentiating
\[
A(k) = \frac{1}{2\pi}\int_0^{2\pi}\left(\sigma k \uh_z^2 - 2k^3\uh_{zz}^2\right)\,\rd z\,.
\]
The derivatives of $A(k)$ are recorded in Appendix \ref{app-a}.  The Jordan chain
theory from \S\ref{subsec-floquet} gives that the second order solution is
$w_2=q_X\xi_3 + c_2\xi_1$ where $c_2$ is an arbitrary constant.

At third order, after some calculation we find
\[
{\bf L}(0)\Big(w_3 - (\xi_3)_kqq_X - \xi_4 q_{XX}\Big)=0\,,
\]
which is trivially solvable with solution,
\[
w_3 = (\xi_3)_kqq_X + \xi_4 q_{XX} + c_3\xi_1\,,
\]
where $c_3$ is an arbitrary constant.

The solvability condition at fourth order is what generates the normal form. The
complete fourth order equation is
\begin{equation}\label{eps-four-eqn}
{\bf L}(0)w_4 + \Upsilon q^2q_X + \mathcal{T}q_{XXX} +\textsf{Solv}=0\,,
\end{equation}
where $\textsf{Solv}$ are terms that vanish identically in the solvability
condition, and the two key coefficients are
\[
\begin{array}{rcl}
\Upsilon &=&
 2k^3\uh_{zzzkkk}
+ 15k^2\uh_{zzzkk} +24k\uh_{zzzk} +4k^3 (\xi_3)_{zzzzk} \\[2mm]
&&\quad  + 6\uh_{zzz}  + 6k^2  \big(\xi_3\big)_{zzzz} \\[2mm]
&&\quad  
 + \sigma k \uh_{zkkk}  + \frac{1}{2}\sigma \uh_{zkk}  +2\sigma \uh_{zkk} 
 + 2\sigma k (\xi_3)_{kzz} + \sigma  (\xi_3)_{zz} \\[2mm]
&&\quad +\frac{1}{2}V''''(\uh)\uh_k^2\xi_3
+\frac{1}{2}V'''(\uh)\uh_{kk}\xi_3
+ V'''(\uh)\uh_k (\xi_3)_k  \,,
\end{array}
\]
and
\[
\mathcal{T}=
4k^3 (\xi_4)_{zzz} + 2\sigma k (\xi_4)_z
+6k^2(\xi_3)_{zz} +  \sigma \xi_3
+ 4k \uh_{zk}  +  \uh_z \,.
\]
Applying the solvability condition then gives
\[
\mathscr{K} q_{XXX} + \kappa q^2q_X = 0 \,,
\]
with
\[
\begin{array}{rcl}
\kappa &=& \lth \uh_z , 2k^3\uh_{zzzkkk}
+ 15k^2\uh_{zzzkk} +24k\uh_{zzzk} +4k^3 (\xi_3)_{zzzzk} \\[2mm]
&&\quad  + 6\uh_{zzz}  + 6k^2  \big(\xi_3\big)_{zzzz} \\[2mm]
&&\quad  
 + \sigma k \uh_{zkkk}  + \frac{1}{2}\sigma \uh_{zkk}  +2\sigma \uh_{zkk} 
 + 2\sigma k (\xi_3)_{kzz} + \sigma  (\xi_3)_{zz} \\[2mm]
&&\quad +\frac{1}{2}V''''(\uh)\uh_k^2\xi_3
+\frac{1}{2}V'''(\uh)\uh_{kk}\xi_3
+ V'''(\uh)\uh_k (\xi_3)_k  \rth\,.
\end{array}
\]
The other coefficient $\mathscr{K}$ is
\[
\mathscr{K} = \lth\uh_z,\mathcal{T}\rth =\lth \uh_z,
4k^3 (\xi_4)_{zzz} + 2\sigma k (\xi_4)_z
+6k^2(\xi_3)_{zz} +  \sigma \xi_3
+ 4k \uh_{zk}  +  \uh_z\rth\,.
\]
The major challenge is now to prove that $\kappa=\fr A_{kkk}$ and this is given
in Appendix \ref{app-c}. The parameter $\mathscr{K}$ can be interpreted various
ways: it can be related to the condition for termination of the Jordan
chain in (\ref{chain-4}), or characterized in terms of derivatives of the Bloch coefficient (see Appendix \ref{app:bloch}).

 \renewcommand{\theequation}{C-\arabic{equation}}

\section{Proof that \texorpdfstring{$2\kappa$} ~~equals \texorpdfstring{$A_{kkk}$} ~~in normal form}
\label{app-c}
\setcounter{equation}{0}

In this appendix the proof is sketched that the coefficient of $q^2q_X$ in the normal form,
denoted by $\kappa$, can be expressed in terms of $A_{kkk}$. Details are in the supplementary material.  Formulas for $A$ and $A_{kkk}$ are given in Appendix \ref{app-a}.

Denote by $\Upsilon$ the vector-valued coefficient of $q^2q_X$ in the $\eps^4$ equation (see equation (\ref{eps-four-eqn})),
\[
\begin{array}{rcl}
\Upsilon &=&
 2k^3\uh_{zzzkkk}
+ 15k^2\uh_{zzzkk} +24k\uh_{zzzk} + 6\uh_{zzz} \\[2mm]
&&\quad +4k^3 (\xi_3)_{zzzzk}  + 6k^2  \big(\xi_3\big)_{zzzz} 
+ 2\sigma k (\xi_3)_{kzz} + \sigma  (\xi_3)_{zz} \\[2mm]
&&\quad  
 + \sigma k \uh_{zkkk}  + \frac{5}{2}\sigma \uh_{zkk}   
  \\[2mm]
&&\quad +\frac{1}{2}V''''(\uh)\uh_k^2\xi_3
+\frac{1}{2}V'''(\uh)\uh_{kk}\xi_3
+ V'''(\uh)\uh_k (\xi_3)_k  \,.
\end{array}
\]
Applying the solvability condition, the scalar-valued coefficient $\lth\uh_z,\Upsilon\rth$ of $q^2q_X$ is then
\[
\begin{array}{rcl}
\kappa &=& \lth \uh_z , 2k^3\uh_{zzzkkk}
+ 15k^2\uh_{zzzkk} +24k\uh_{zzzk}  + 6\uh_{zzz}  \\[2mm]
&&\quad  
 + \sigma k \uh_{zkkk}  + \frac{5}{2}\sigma \uh_{zkk}  
 \\[2mm]
&&\quad +4k^3 (\xi_3)_{zzzzk} + 6k^2  \big(\xi_3\big)_{zzzz}
 + 2\sigma k (\xi_3)_{kzz} + \sigma  (\xi_3)_{zz} \\[2mm]
&&\quad +\frac{1}{2}V''''(\uh)\uh_k^2\xi_3
+\frac{1}{2}V'''(\uh)\uh_{kk}\xi_3
+ V'''(\uh)\uh_k (\xi_3)_k  \rth\,.
\end{array}
\]
The aim is to show that $2\kappa = A_{kkk}$, where $A_{kkk}$ is given in
Appendix A. We note that $A_{kkk}$ is a quadratic functional and can therefore be written in terms of the inner product.  This will be useful when comparing it with the form of $\kappa$. The inner product based expression for $A_{kkk}$, evaluated on a family of periodic orbits, is
\begin{equation}\label{AKKK-formula}
\begin{array}{rcl}
A_{kkk}  &=&\displaystyle
6\sigma \lth \uh_{zk},\uh_{zk}\rth +6\sigma \lth \uh_z,\uh_{zkk}\rth 
 + 6\sigma k \lth \uh_{zk},\uh_{zkk}\rth  + 2\sigma k \lth \uh_z,\uh_{zkkk}\rth\\[4mm]
&&\qquad\qquad -12 \lth \uh_{zz},\uh_{zz}\rth
 -72 k\lth \uh_{zz},\uh_{zzk}\rth  -36k^2\lth \uh_{zzk},\uh_{zzk}\rth  - 36 k^2\lth \uh_{zz},\uh_{zzkk}\rth \\[4mm]
&&\qquad\qquad\qquad - 12k^3\lth \uh_{zzkk},\uh_{zzk}\rth  - 4k^3\lth \uh_{zz},u_{zzkkk}\rth\,.
\end{array}
\end{equation}
In $A_{kkk}$ are $k$ and $z$ derivatives of $\uh$, whereas in $\kappa$ we find terms with $\xi_3$ as well as derivatives with respect to $z$ and $k$ of $\xi_3$.  Starting with the Jordan chain equation for $\xi_3$ and differentiating, we find equations for $\xi_3$, $(\xi_3)_k$, $(\xi_3)_z$, and so on.  Back substituting all these equations, integrating by parts, and simplifying gives the result.  The details are in the supplementary material.

 \renewcommand{\theequation}{D-\arabic{equation}}

\section{Connection between \texorpdfstring{$\mathscr{K}$} ~and the Bloch spectrum}\label{app:bloch}

This appendix is concerned with connecting the coefficient of the spatial
derivative with a quantity that can be computed independently of the asymptotic
analysis, which transpires to be the Bloch spectrum of the original wavetrain
$\hat{u}$. To observe this, we follow previous work on phase dynamics
\cite{doelman2009dynamics,ratliff2021} and construct the Bloch spectral problem
for (\ref{primary-ode}):
\begin{equation}
    {\bf L}(\nu) = \left(i\nu+k\partial _z\right)^4+\sigma \left(i \nu+k\partial _z\right)^2+V''(\hat{u})
\end{equation}
Whilst the main paper focuses on the Goldstein mode to this operator (i.e.
exactly at $\nu = 0$), we will instead consider the generalization where the
operator admits a continuous spectrum near zero:
\begin{equation}\label{eqn:eigen-problem}
{\bf L}(\nu) W(z,\nu) = \mu(\nu)W(z,\nu)\,.
\end{equation}
Here we assume the eigenvalue $\mu(\nu)$ is at least 4 times differentiable in the neighborhood of
$\nu=0$. The relationship between the Bloch spectrum $\mu(\nu)$ and
$\mathscr{K}$ is established by taking $\nu-$derivatives and relating these derivatives to inner products of the basic
wave. Since ${\bf L}(0)$ and the original system are invariant under $z\to-z$
transformations, it follows that the Bloch spectrum should also be invariant
under $\nu \to -\nu$ transformations. This leads to the observation that the
spectrum $\mu(\nu)$ is an even function of $\nu$, so $\mu'(0) = \mu '''(0) = 0$.
Furthermore, ${\bf L}(0)$ is self adjoint. Finally, given the coincidence of
${\bf L}(0)$ with the linear operator of the original system extracted from a
$z$ derivative of (\ref{primary-ode}) we have  that $\mu(0) = 0$ and $W(z,0) =
\hat{u}_z\,.$ This means that $\hat{u}_z$ is also the adjoint eigenfunction by
assumption.

We now seek a relationship between $\mu$ and the basic state. The first
derivative of (\ref{eqn:eigen-problem}) evaluated at zero gives
\[
{\bf L}(0) \partial_\nu W(z,0) = \mu'(0) \hat{u}_z-{\bf L}'(0)\hat{u}_z = -(4k^3\hat{u}_{zzzz}+2k\sigma \hat{u}_{zz})\,,
\]
where the primes herein denote differentiation with respect to $\nu$.
Comparisons to a $k$ derivative of the equation for the basic state and the
above expression shows that
\[
\partial_\nu W(z,0) = i\hat{u}_k\,.
\]
Taking a second derivative of (\ref{eqn:eigen-problem}) and evaluating it at
zero gives the next order problem:
\[
{\bf L}(0)\partial_\nu^2 W(z,0) = \mu''(0)\hat{u}_z-2i{\bf L}'(0)\hat{u}_k-{\bf L}''(0)\hat{u}_z = \mu''(0)\hat{u}_z+8k^3\hat{u}_{zzzk}+4k\sigma \hat{u}_{zk}+12k^2\hat{u}_{zzz}+2\sigma \hat{u}_z\,.
\]
This system is solvable, which is required as is its the derivative of a strict
equality, when the right hand side vanishes under the inner product established
in the main paper. We can see that
\[
\begin{split}
\mu''(0) &= -\frac{\langang \hat{u}_z,8k^3\hat{u}_{zzzk}+4k\sigma \hat{u}_{zk}+12k^2\hat{u}_{zzz}+2\sigma \hat{u}_z\rangang}{\langang \hat{u}_z,\hat{u}_z\rangang}\\[3mm]
=& -\frac{1}{2\pi \langang \hat{u}_z,\hat{u}_z\rangang}\int_0^{2\pi}2\sigma \hat{u}_z^2+4\sigma \hat{u}_z\hat{u}_{zk}-12k^2\hat{u}_{zz}^2-8k^3\hat{u}_{zz}\hat{u}_{zzk} \, dz= -\frac{2 A_k}{\langang \hat{u}_z,\hat{u}_z\rangang}\,.
\end{split}
\]
We note that the denominator here can be removed by redefining $W(z,0) =
\hat{u}_z/\langang \hat{u}_z,\hat{u}_z\rangang$. 
The condition assumed in the main paper is that $A_k = 0$, meaning this
derivative of $\mu$ vanishes and thus the equation at this order is solvable
with
\[
\partial_\nu^2W(z,0) = -2\xi_3\,.
\]

This will be the final parts of our analysis. A third derivative, with all our
previous results, gives
\[
{\bf L}\partial_\nu^3W(z,0) = -\left[{\bf L}'''(0)\hat{u}_z+3i{\bf L}''(0)\hat{u}_k-3{\bf L}'(0)\xi_3\right]
\]
which from the main paper gives that $\partial_\nu^3W(z,0) = -6i\xi_4$. One
final derivative gives our final order, where the key result emerges:
\[
\begin{split}
{\bf L}\partial_\lambda^4 W(z,0) &= \mu^{(4)}(0)\hat{u}_z-\left[{\bf L}''''(0)\hat{u}_z+4i{\bf L}'''(0)\hat{u}_k-6{\bf L}''(0)(2\xi_3)-4i{\bf L}'(0)(6\xi_4)\right]\\[3mm]
&=\mu^{(4)}(0)\hat{u}_z-24\left[ \hat{u}_z+4k\hat{u}_{zk}+(6k^2\partial_{zz}+\sigma)\xi_3+(4k^3\partial_z^3+2\sigma k\partial_z)\xi_4\right]
\end{split}
\]
Imposing solvability and using the definition of $\mathscr{K}$ gives our desired
result:
\[
\mu''''(0) = \frac{24 \mathscr{K}}{\langang \hat{u}_z,\hat{u}_z\rangang}\,,
\]
with the observation that one can again remove the denominator by redefining
$W(z,0)$. 
\end{appendix}

 \renewcommand{\theequation}{E-\arabic{equation}}

\section{Underpinning symplectic geometry}
\setcounter{equation}{0}
\label{sec-symplectic}

The properties and implications of action become clearer when the symplectic
structure of the governing equations is revealed.  We already showed that the
first-order form (\ref{first-order-form}) is Hamiltonian with a non-canonical
symplectic operator ${\bf J}$, defined in (\ref{J-def}). We now transform the coordinates $(u_1,u_2,u_3,u_4)$ to canonical $(q_1,q_2,p_1,p_2)$ via
  \[
  \begin{pmatrix} q_1\\ q_2 \\ p_1\\ p_2 \end{pmatrix}
  = \left[\begin{matrix} 1 & 0 & 0 & 0 \\ 0 & 0 & 1 & 0 \\
      0 & \sigma & 0 & 1 \\ 0 & 1 & 0 & 0 \end{matrix}\right]
  \begin{pmatrix} u_1\\ u_2 \\ u_3\\ u_4 \end{pmatrix}\,.
  \]
  This transformation, denoted by ${\bf T}$ is symplectic in the sense that
  ${\bf T}^T{\bf J}_C{\bf T} = {\bf J}$, where ${\bf J}_C$ is the canonical
  symplectic operator
  \[
  {\bf J}_C :=
  \left[\begin{matrix} 0 & 0 & -1 & 0 \\ 0 & 0 & 0 & -1 \\
      1 & 0 & 0 & 0 \\ 0 & 1 & 0 & 0 \end{matrix}\right]\,.
  \]
  The first-order form (\ref{first-order-form}) is then transformed to
  \begin{equation}\label{canon-ham}
  -{\bf p}_x= \partial_{\bf q}H \qand {\bf q}_x = \partial_{\bf p}H\,,
    \end{equation}
  with Hamiltonian function,
  \[
  H({\bf q},{\bf p}) = p_1p_2 - \frac{\sigma}{2}p_2^2 - \frac{1}{2} q_2^2 + V(q_1)\,.
  \]
  In canonical coordinates, the action is based on the one form ${\bf
  p}\cdot\rd{\bf q}$.  Let $({\bf q}(x,s),{\bf p}(x,s))$ be solutions of
  (\ref{canon-ham}) parameterized by $s$ and $2\pi-$periodic in $s$, as in
  (\ref{action-def}).  Then 
  \[
  \frac{d\ }{dx} \oint {\bf p}\cdot{\bf q}_s\,\rd s = 0 \,,
  \]
  which can be confirmed by direction calculation.  This is in fact the
  Poincar\'e-Cartan theorem, which can be found in most textbooks on Hamiltonian
  dynamics. Figure \ref{fig-action-cylinder} shows a schematic of the theorem in
  action.  Back substituting for ${\bf q}$ and ${\bf p}$,
  \[
  p_1(q_1)_s + p_2(q_2)_s = (u_4+\sigma u_2) (u_1)_s + u_2(u_3)_s = (u_{xxx}+
  \sigma u_x)u_s + u_xu_{xxs}\,,
  \]
  recovering the definition
  of action in (\ref{action-def}).

  A second result which becomes clearer in the Hamiltonian setting is the
  formula $H_k = kA_k$ proved in Lemma \ref{lemma-Hk}.  Here it will follow from
  a constrained variational principle.  Let $z=kx$ and substitute into
  (\ref{canon-ham}), giving
  \begin{equation}\label{canon-ham-k}
  k\underbrace{\left[\begin{matrix} {\bf 0} & -{\bf I}\\ {\bf I} & {\bf 0}\end{matrix}\right] \begin{pmatrix}{\bf q}\\ {\bf p} \end{pmatrix}_z }_{\displaystyle \nabla A}= \underbrace{\begin{pmatrix} \partial_{\bf q}H\\ \partial_{\bf p}H\end{pmatrix}}_{\displaystyle \nabla H}\,,
  \end{equation}
  In other words, $2\pi-$ periodic solutions can be characterized as critical
  points of $H$ on level sets of the action, with $k$ as Lagrange multiplier,
  and (\ref{canon-ham-k}) as the Lagrange necessary condition; that is, $\nabla
  H=k\nabla A$.  Given a solution of this variational principle, it follows that
  \[
  H_k = \lth\nabla H,U_k\rth = \lth\nabla H - k \nabla A,U_k\rth + k\lth \nabla A, U\rth = k A_k\,,
  \]
  on solutions, with $U:=({\bf q},{\bf p})$, giving a concise proof of Lemma \ref{lemma-Hk}, where $\lth\cdot,\cdot\rth$ is
  an inner product for $2\pi-$periodic vector valued functions, generalizing
  (\ref{ip-def}).  

  Since the system (\ref{canon-ham})is first order, the Jordan chain theory of \S\ref{sec-periodic-states} is for the linear eigenvalue problem
  \[
  \big[ D^2 H - kD^2 A\big]{\bf v} = \lambda {\bf J}_C{\bf v}\,.
  \]
  The left-hand side has a symmetric operator and the right-hand side has a skew-symmetric operator.  The Jordan chain, for an eigenvalue $\lambda_*$ satisfies
  \[
  {\bf L}{\bf v}_{j+1} = \lambda_*{\bf J}_C{\bf v}_{j+1} + {\bf J}{\bf v}_{j}\,,\quad j=1,2,\ldots\,,\quad \mbox{with}\ {\bf L} =D^2H - k{\bf J}_C\frac{d\ }{dz}\,.
\]

In the symplectic setting, other new variational characterizations appear.  For example, in the case of SH357, all the parameters in $H$ appear linearly, and so it can be split into a sum of   functionals
  \[
  H({\bf q},{\bf p}) = H_0 - \mu H_1 - a H_2 - b H_3\,.
  \]
  In this case the Lagrange
  necessary condition, generalizing (\ref{canon-ham-k}), is
  \[
  \nabla H_0 = k\nabla A + \mu \nabla H_1 + a\nabla H_2+b\nabla H_3\,,
  \]
  with $k,\mu,a,b$ now considered as Lagrange multipliers.  However, we have not
  as yet found any advantage to this more general variational principle.

    The symplectic form of the equations also
    simplifies the phase modulation.  For example, without any structure, the
    formula relating $A_{kkk}$ to the normal form coefficient $\kappa$ in Appendix \ref{app-c} is quite surprising. Using
    symplectic phase modulation (e.g.\ Bridges \cite{bridges2017symmetry}), it
    can be reduced to a few terms.  Indeed, we did the calculation first from a
    symplectic viewpoint and then transformed coordinates to apply to the setting of the scalar-valued ODE (\ref{primary-ode}).

\bibliographystyle{acm}
\bibliography{Heteroclinic_bib}

\end{document}